\documentclass[10pt,twoside, a4paper, english, reqno]{amsart}
\usepackage[dvips]{epsfig}
\usepackage{amscd}
\usepackage{stmaryrd}
\usepackage{amssymb}
\usepackage{amsthm}
\usepackage{amsmath}
\usepackage{graphicx,enumitem,dsfont}
\usepackage{latexsym}

\usepackage{upref}
\usepackage[colorlinks]{hyperref}
\usepackage{color}
\usepackage{subcaption}
\usepackage[usenames,dvipsnames]{xcolor}
\theoremstyle{plain}
\newtheorem{thm}{Theorem}[section]
\theoremstyle{plain}
\newtheorem{lem}[thm]{Lemma}
\newtheorem{prop}[thm]{Proposition}
\newtheorem{cor}[thm]{Corollary}

\theoremstyle{definition}
\newtheorem{defi}{Definition}[section]
\newtheorem{rem}{Remark}[section]

\newtheorem*{maintheorem*}{Main Theorem}

\newenvironment{Assumptions}
{
\setcounter{enumi}{0}

\begin{enumerate}}
{\end{enumerate} }

\newenvironment{Assumptions2}
{
\setcounter{enumi}{0}

\begin{enumerate}}
{\end{enumerate} }

\newcommand{\R}{\ensuremath{\mathbb{R}}}

\newcommand{\goto}{\ensuremath{\rightarrow}}

\newcommand{\eps}{\ensuremath{\varepsilon}}

\numberwithin{equation}{section} \allowdisplaybreaks

\title[Large Deviations Principle]
{Stochastic evolutionary $p$-Laplace equation: Large Deviation Principles and Transportation Cost Inequality}

\date{}

\author[ Kavin]{Kavin R}
\address[Kavin R] {\newline 
Department of Mathematics,
Indian Institute of Technology Delhi,
Hauz Khas, New Delhi, 110016, India.}
\email[] {maz198757@iitd.ac.in}

\author[A. K. Majee]{Ananta K. Majee}
\address[Ananta K. Majee]{\newline
Department of Mathematics,
Indian Institute of Technology Delhi,
Hauz Khas, New Delhi, 110016, India. }
\email[]{majee@maths.iitd.ac.in}


\thanks{}

\numberwithin{equation}{section} \allowdisplaybreaks
\begin{document}
\begin{abstract}
    In this paper, we establish large deviation 
principle for the strong solution of evolutionary $p$-Laplace equation driven by small multiplicative Brownian noise, where the weak convergence approach plays a key role. Moreover, by using Girsanov transformation  along with $L^1$-contraction approach, we show the quadratic transportation cost inequality for the strong solution of the underlying problem.

\end{abstract}
\maketitle
\section {Introduction}
Let $(\Omega, \mathbb{P}, \mathcal{F}, \{ \mathcal{F}_t\}_{t \geq 0})$
is a filtered probability space satisfying the usual hypothesis. We are interested in the theory of large deviation principle, which concerns the asymptotic behaviour of remote tails of some sequence of probability distribution, for the solution of the following evolutionary $p$-Laplace equation perturbed by multiplicative  Brownian noise:
\begin{equation}\label{problem1.1}
      \begin{aligned}
            du  - {\rm div}_x ( |\nabla u|^{p-2}\nabla u +\vec{f}(u)) \, dt &= H(u)\, dW \quad  &in& \ \ \Omega \times D_{(0,T)}, \\
             u &= 0 \quad &on& \ \ \Omega \times \partial D_{(0,T)}, \\
            u(0,.) &= u_0(.) \in L^2(D) &in& \ \ \Omega \times D,
        \end{aligned}
\end{equation}
where $p > 2$, $ D_{(0,T)}: = (0,T) \times D$ with $T > 0$ fixed,  $D \subset \mathbb{R}^d$ is a bounded domain with Lipschitz boundary $\partial D$, and $ \partial D_{(0,T)}:=(0,T)\times \partial D$. Moreover,
$\vec{f} : \mathbb{R} \rightarrow \mathbb{R}^d$ is a  given flux function and $ H : \mathbb{R} \rightarrow \mathbb{R}$ is a given diffusion function.  Furthermore, $W$ is a $1$-dimensional Brownian motion defined on the  filtered probability space $(\Omega, \mathbb{P}, \mathcal{F}, \{ \mathcal{F}_t\}_{t \geq 0})$.
\vspace{0.1cm}

 The problem \eqref{problem1.1} could be viewed as evolutionary $p$-Laplace equation perturbed by nonlinear flux function and Brownian noise. This type of equation makes an appearance in the field of physics, mechanics and biology \cite{dib,wuyin}. The presence of non-linearity in the drift  and diffusion terms in equation prevents us to define a semi-group solution \cite{prato}. In the theory of stochastic evolution equations two variety of solution concept are considered specifically strong solution and weak solution. In \cite{gv1}, the authors have established well-posedness theory for the strong solution of  \eqref{problem1.1}. By using the  techniques of  semi-implicit time discretization, Skorokhod and Prokhorov theorem, they first established existence of martingale/weak solution.  Moreover, employing the Gy{\"o}ngy characterization of convergence in probability and pathwise uniqueness (which was established by $L^1$- contraction method) of weak solution,  the authors established well-posedness of strong solution of \eqref{problem1.1} in the sense of probability.  We also refer \cite{maj}, where well-posedness theory of martingale solution of the evolutionary $p$-Laplace equation driven by multiplicative L{\'e}vy noise was established. 

\vspace{0.2cm}

Theory of large deviation principle~(LDP in short)  makes an appearance in the field of probability and statistics; see  \cite{varadha1,varadhan, freidlin, stroock, dembo, Ellis-1985, Stroock-1989} and reference therein. In the past three decades, there are quite a few results  available over LDP for stochastic partial differential equations~(SPDEs); see e.g., \cite{liu, dong1, dong2, matous, rock, ren} and references therein, based on either weak convergence method \cite{dupuis, budhi} or 
exponential equivalence method. Using exponentially equivalent technique the authors, in \cite{rock}, established LDP for the solution of the stochastic tamed Navier-Stokes equations. In \cite{ren}, Ren and Zhang studied Freidlin-Wentzell's large deviation for stochastic evolution equation in the evolution triple. LDP  for stochastic evolution equation involving general monotone drift  was studied by Liu in \cite{liu}. The study of LDP was carried out for 3D stochastic primitive equation by Dong, Zhai and Zhang in \cite{dong1}, stochastic scalar conservation laws by Dong, Wu and Zhang in \cite{dong2}, Obstacle problems of quasi-linear SPDEs by Matoussi, Sabbagh and Zhang in \cite{matous} via weak convergence method.

 \vspace{0.2cm}

For problem \eqref{problem1.1},  exponential equivalence method is not applicable, due to lac of existence result of LDP of the corresponding auxiliary problem in the appropriate solution space, instead we use the weak convergence procedure \cite{budhi}. Because of nonlinear perturbation ${\rm div}_{x}f(u)$ of $p$-Laplace operator $(p > 2)$, we are not able to use the results of monotone or locally monotone SPDEs (cf.~\cite{liurock,pard}) to establish LDP---which  differs from the paper \cite{liu}.  We settle this issue by employing semi-discrete time discretization together with a-priori estimation on some appropriate fractional Sobolev space . 

\vspace{0.2cm}
The study of  transportation cost inequality ~(TCI in short) is an important area of research as it has close connections with well-known functional inequalities, e.g., poincar{\'e} inequalities, logarithmic Sobolev inequalities.  The $L^1$-transportation cost inequality was first introduced in \cite{Tal2}. The connection between $L^1$-transportation cost inequality and normal concentration was studied in \cite{Tal, Marton-1996-1, Marton-1996-2, guilli}.  In  \cite{otto}, the authors studied relation between quadratic TCI and other functional inequalities; see also \cite{Ledoux-2001}.  For general theory of TCI, we refer the reader to \cite{villan}.
 In the recent past, many authors are devoted to study quadratic TCI for SPDEs  \cite{bouf, zhang,shang,sarantsev-2019}. Quadratic TCI for more general SPDEs was established in \cite{sarantsev-2019} under the $L^2$  distance and uniform distance in case of additive noise. In \cite{bouf}, the authors have shown the quadratic TCI for stochastic heat equations perturbed by multiplicative space-time white noise under the uniform norm. Very recently, in \cite{zhang}, the authors have obtained Quadratic TCI for the  kinetic solution of stochastic conservation laws driven by Brownian noise.
  
 \vspace{0.2cm}

The aim of this paper is twofold:
\begin{itemize}
 \item[i)] Firstly, we show the small perturbation LDP for the strong solution of \eqref{problem1.1} on the solution space $\mathcal{Z} = C([0,T];L^{2}(D))\cap L^{p}([0,T];W_{0}^{1,p}(D))$ via the weak convergence method. We first prove the well-posedness result of the corresponding skeleton equation~(cf.~\eqref{eq:skeleton}) based on the  technique of time discretization. We built a sequence of approximate solution and avail of {\em a-priori} estimates, compact embedding to prove existence of skeleton equation. Moreover, we use standard $L^1$-contraction approach to have uniqueness of skeleton equation. Furthermore, by using a-priori estimates for skeleton equation and the $v_\eps$, solution of the equation \eqref{eq:epsilon-main}, together with Girsanov and Skorokhod theorem, we derive LDP property of the strong solution of \eqref{problem1.1} on $\mathcal{Z}$.
\item[ii)] Secondly, we establish the quadratic TCI for the solution of \eqref{problem1.1} on $L^2(0,T; L^1(D))$ based on the method of Girsanov transformation and $L^1$-contraction principle. 
\end{itemize}
\vspace{0.2cm}

The rest of the paper is organized as follows. we discuss the assumption, basic definition and framework in Section \ref{sec:preliminaries}. In Section \ref{sec:sk}, we show well-posedness result of skeleton equation.  Section \ref{sec:LDP} is devoted to prove the LDP for the strong solution of \eqref{problem1.1} . In the final section \ref{sec:TCI}, we prove  quadratic TCI for the solution of \eqref{problem1.1}. 

\section{Preliminaries and Technical Framework} \label{sec:preliminaries}

Let $p'$ denotes the convex conjugate of $p$. Consider the Gelfand triple 
$$ W_{0}^{1,p}(D) \hookrightarrow L^{2}(D) \equiv (L^{2}(D))^{'} \hookrightarrow W^{-1,p'}(D)  $$ (see \cite[Section 7.2]{roubi}). 
We denote by $(\cdot, \cdot)_{L^2(D)}$ and 
 $\left< \cdot, \cdot  \right >_{W^{-1,p'}(D),W_{0}^{1,p}(D)}$, the usual scalar product on $L^2(D)$ and dualization between $W^{-1,p'}(D)$ and $W_{0}^{1,p}(D)$ respectively. For simplicity, we rewrite (in some place) these as  $(\cdot, \cdot)$ and   $\left< \cdot, \cdot  \right >$. Throughout this paper, we use the letters $C,\,L$ etc to denote various generic constants. The $p$-Laplacian operator, denoted by $\Delta_p,$ is defined  by
$$\Delta_{p}u := \nabla\cdot(|\nabla{u}|^{p-2} \nabla{u}), \quad 1 < p <\infty.$$

 To prove LDP property and quadratic TCI for the solution of \eqref{problem1.1}, we made the following assumptions:
  \begin{Assumptions}
\item \label{A1} The initial function $u_0\in L^2(D)$.
\item \label{A2} The flux function $\vec{f} : \mathbb{R} \rightarrow \mathbb{R}^d$ is a Lipschitz continuous with $\vec{f}(0) = \vec{0}.$
\item \label{A3} The diffusion coefficient $ H : \mathbb{R} \rightarrow \mathbb{R}$ is  a Lipschitz continuous function with $H(0) = 0.$
 \end{Assumptions}

\subsection{Large deviation principle}  Recall some basic definitions and results of large deviation theory. Let $\{X^\epsilon\}_{\epsilon > 0}$  be a sequence of random variable defined on a probability space $(\Omega, \mathcal{F}, \mathbb{P})$ taking values in Polish space $\mathcal{Z}$.
     \begin{defi}[Rate Function]
      A function I : $\mathcal{Z}$ $\rightarrow$ $[0,\infty]$ is called a rate function if I is lower semi-continuous. A rate function I is called a good rate function if the level set \{x $\in$  $\mathcal{Z}$ : I(x) $\leq$ M \} is compact for each M $<$ $\infty$.
\end{defi}
\begin{defi} [Large deviation principle]
The sequence  $\{X^\epsilon\}_{\epsilon > 0}$  is said to satisfy the large deviation principle with rate function I if for each Borel subset A of $\mathcal{Z}$ 
\begin{equation*}
    \begin{aligned}
       - \underset{x\in A^0} \inf I(x) \leq \underset{\epsilon \longrightarrow 0}{\liminf} \big(\epsilon^{2}\log\mathbb{P}(X^\epsilon \in A)\big) & \leq    \underset{\epsilon \longrightarrow 0}{\limsup} \big(\epsilon^{2} \log \mathbb{P}(X^\epsilon \in A)\big)
        \leq  - \underset{x\in \bar{A}} \inf I(x),
    \end{aligned}
\end{equation*}
where $A^0$ and $\bar{A}$ are interior and closure of $A$ in $\mathcal{Z}$ respectively. 
\end{defi}
Let $\left\{W(t) \right\}_{t \geq 0}$ be a real-valued Wiener process on a given complete probability space $\big(\Omega, \mathcal{F}, \mathbb{P}, \{ \mathcal{F}_t\}_{t \geq 0} \big)$.  Consider the following sets:
\begin{align}
      &\mathcal{A} := \Big\{ \phi :  \text{ $\phi$  is a  real-valued  $\{\mathcal{F}_t\}$-predictable process  such that } \notag\\  
      & \hspace{4cm} \int_0^T|\phi(s)|_{\mathbb{R}}^2ds <  \infty,  \, \mathbb{P}-a.s.\Big\}, \notag \\
      &S_M := \Big\{h \in L^2([0,T]; \mathbb{R}) : \int_0^T|h(s)|_{\mathbb{R}}^2\,ds\leq M\Big\}, \notag \\
    &\mathcal{A}_M :=\Big\{\phi \in \mathcal{A} : \phi(\omega) \in S_M,  \mathbb{P}-a.s.\Big\}. \notag
\end{align}
Suppose for each  $\epsilon > 0$, $\mathcal{G^{\epsilon}}$ : C([0,T], $\mathbb{R}$) $\rightarrow$ $\mathcal{Z}$ is a measurable map and $X^{\epsilon}$ := $\mathcal{G}^{\epsilon}(W)$. Then,  the sequence $X^\epsilon$ satisfies the LDP on $\mathcal{Z}$ if the following sufficient conditions hold: 
\vspace{0.1cm}

\noindent{\bf Condition~C}:  There exists a measurable function  $\mathcal{G}^{0} : C([0,T]; \mathbb{R}) \rightarrow \mathcal{Z}$ such that the following conditions hold.
\begin{Assumptions2}
\item \label{C1}  Let \{$h^\epsilon$; $\epsilon$ $>$ 0\} $\subset$ $\mathcal{A}_M$, for some $M < \infty$, be a family of process converges to $h$ in distribution as $S_{M}$-valued random elements.  Then  
\begin{align*}
    \mathcal{G}^{\epsilon} \bigg( W(\cdot) + \frac{1}{{\epsilon}} \int_0^{\cdot} h^{\epsilon}(s)\,ds\bigg) \rightarrow  \mathcal{G}^0 \bigg( \int_0^{\cdot}h(s) \, ds \bigg)
\end{align*}
in distribution as $\epsilon \rightarrow 0$.
\item \label{C2} For every M $<$ $\infty$, the set
\begin{align*} 
    K_{M} = \left \{\mathcal{G}^0\bigg(\int_0^{.}h(s)\,ds \bigg) : h \in S_{M} \  \right \}
\end{align*}
is a compact subset of $\mathcal{Z}$.
\end{Assumptions2}

The following lemma is due to Budhiraja and Dupuis in \cite[Theorem 4.4]{budhi}.
\begin{lem} \label{lem:main}
 If $X^{\epsilon}$ := $\mathcal{G}^{\epsilon}(W)$ satisfies {\bf  Condition C}, then $X^\epsilon$ satisfies the LDP on $\mathcal{Z}$ with the good rate function $I$, defined by
 \begin{align} \label{eq:Imain}
  I(f)= \begin{cases}
\displaystyle  \underset{\big\{h \in L^2([0,T]; \mathbb{R}) : f =  \mathcal{G}^{0} \big(\int_0^{\cdot}h(s)ds\big)\big\}}\inf  \bigg \{ \frac{1}{2} \int_0^T|h(s)|_{\mathbb{R}}^2ds \bigg \}, \quad \forall f \in \mathcal{Z}\,,
 \\
\displaystyle   \infty, \quad if \quad \bigg \{h \in L^2([0,T]; \mathbb{R}) : f = \mathcal{G}^0 \bigg( \int_0^{.}h(s)ds \bigg) \bigg \} = \emptyset.
\end{cases}
 \end{align}
\end{lem}
  
 For $\epsilon>0$,  consider the equation driven by small multiplicative noise:
 \begin{equation}
 \label{eq:ldp}
      \begin{aligned}
            du^{\epsilon}  - {\rm div}_x ( |\nabla u^{\epsilon}|^{p-2}\nabla u^{\epsilon} + \vec{f}(u^{\epsilon})) \, dt &= \epsilon\, H(u^{\epsilon})\, dW \quad  &in& \ \ \Omega \times D_{(0,T)}, \\
            u^{\epsilon}(0,.) &= u_0(.) \in L^2(D) &in& \ \ \Omega \times D\,.
        \end{aligned}
\end{equation}

Under the assumptions \ref{A1}-\ref{A3}, equation \eqref{eq:ldp}  has a unique strong solution \\
 $u^{\epsilon}\in L^2(\Omega; C([0,T];W^{-1,p'}(D))\cap L^{p}([0,T];W_{0}^{1,p}(D)))$; cf.~ \cite[Theorem 2]{gv1}. Moreover, following the arguments as in \cite[Lemma $4.1$]{Pardoux1974-75}, one can easily see that the path-wise unique solution $u^{\eps}$ of  \eqref{eq:ldp}  has improved regularity property namely $u^{\eps}\in L^2(\Omega;  C([0,T];L^{2}(D))\cap L^{p}([0,T];W_{0}^{1,p}(D)))$. Thus, there exists a Borel-measurable function $\mathcal{G^{\epsilon}} : C([0,T];\mathbb{R}) \rightarrow \mathcal{Z}: = C([0,T];L^{2}(D))\cap L^{p}([0,T];W_{0}^{1,p}(D))$ such that $u^{\epsilon}$ := $\mathcal{G}^{\epsilon}(W)$ a.s.
\vspace{0.1cm}

 For $h \in L^2([0,T];\mathbb{R})$, we consider the following  skeleton equation 
\begin{equation}
\label{eq:skeleton}
   \begin{aligned}
            du_{h}  - {\rm div}_x ( |\nabla u_{h}|^{p-2}\nabla u_{h} +\vec{f}(u_{h})) \, dt &= H(u_{h})\,h(t)\,dt, \\
            u_{h}(0,.) &= u_0(.). 
        \end{aligned}
\end{equation}
In view of the assumptions \ref{A1}-\ref{A3}, equation  \eqref{eq:skeleton} exhibits unique solution $u_h \in C([0,T];L^{2}(D))\cap L^{p}([0,T];W_{0}^{1,p}(D))$ for any given  $h \in L^2([0,T];\mathbb{R})$;  see  Section
\ref{sec:sk}. Define $\mathcal{G}^{0} : C([0,T],\mathbb{R}) \rightarrow  \mathcal{Z}$ by
\begin{equation}
    \mathcal{G}^{0}(\phi) := \begin{cases}
\displaystyle   u_{h} \quad  \text{if $\phi =\int_{0}^{\cdot} h(s) \, ds$  for some $ h \in L^2([0,T];\mathbb{R})$},\\
     0 \quad \text{otherwise}.
    \end{cases}
\end{equation}
Now we state the main result of our paper.
 \begin{thm} \label{thm:main}
  Let $p > 2$ and the assumptions \ref{A1} - \ref{A3} hold true. For any $\epsilon > 0,~u^{\epsilon}$ satisfies the LDP on $\mathcal{Z} = C([0,T];L^{2}(D))\cap L^{p}([0,T];W_{0}^{1,p}(D))$ with the good rate function $I$ given by \eqref{eq:Imain}, where $u^{\epsilon}$ is the unique solution of \eqref{eq:ldp}.
\end{thm}

\begin{rem}
We point out that  the result of Theorem \ref{thm:main} is still valid if we omit the condition $H(0)=0$. i.e., $|H(u(t,x)| \leq C(1+|u(t,x)|) \ \forall  \, (t,x) \in D_{(0,T)}. $
\end{rem}
\vspace{0.1cm}

\subsection{Transportation cost inequality} Let $(X,d)$ be a metric space equipped with the Borel $\sigma$-field $\mathcal{B}$ and $\mathcal{P}(X)$ denotes the set of all probability measure on $(X,d)$.
 \begin{defi} [Wasserstein distance] Let $\mu, \nu \in \mathcal{P}(X)$ and $r \in [1,\infty)$ be given.  We define the $L^r$-Wasserstein distance between  $\mu$ and $\nu$, denoted by $W_r(\nu, \mu) $,  as  (cf.~ \cite{villan} )
 \begin{equation}\label{wd}
 W_r(\nu, \mu) := \bigg [inf \int_{X \times X} d(x,y)^r\pi(dx,dy) \bigg ]^{\frac{1}{r}} \\
 = inf \bigg \{\big[\mathbb{E}(d(\mathcal{X},\mathcal{Y})^r)\big]^{\frac{1}{r}},~~~ law(\mathcal{X}) = \mu, ~~ law(\mathcal{Y}) = \nu \bigg \},
  \end{equation}
 where the infimum is  taken over all the joint probability measure $\pi$ on X $\times$ X  with marginals $\mu$ and $\nu$.
 \end{defi}
 We say that $\nu$ is absolutely continuous with respect to $\mu$, denoted by $\nu  \ll  \mu$, if $\nu(A)=0$ for any  $A\in \mathcal{B}$ such that $\mu(A)=0$. The relative entropy (Kullback information) of $\nu$ with respect to $\mu$ is defined by 
 \begin{equation*}
     \mathcal{H}(\nu | \mu) = \begin{cases}
\displaystyle      \int_X log(\frac{d\nu}{d\mu})d\nu \quad  if \, \nu  \ll  \mu, \\ 
       +\infty \quad \quad \text{otherwise}. 
     \end{cases}
 \end{equation*}
 
\begin{defi}[Transportation cost inequality] \label{5.1.2} 
A measure $\mu$ satisfies the $L^r$-transportation cost inequality if there exist a constant C $>$ 0 such that for all probability measure $\nu$,
\begin{equation} \label{eq:reltci}
W_r(\nu, \mu) \leq \sqrt{2C\mathcal{H}(\nu|\mu)}.
\end{equation}
We denote the relation in \eqref{eq:reltci} as $\mu \in T^{r}(C)$. The case r = 2, $ T^{2}(C)$ is referred to as the quadratic transport cost inequality. 
\end{defi} 
 In view of H\"{o}lder's inequality, one has 
$$ T^{q}(C)\subseteq T^{r}(C)\quad \text{for $1\le r\le q$}.$$
Hence the property  $ T^{2}(C)$ is stronger than  $ T^{1}(C)$.  We now ready to state the main result regarding TCI.
\begin{thm} \label{thm:maintci}
Let the assumptions \ref{A1}-\ref{A3} hold and in addition  the diffusion coefficient $H$ be bounded i.e., there exists a constant $C>0$ such that $\left |H(u(t,x)) \right | \leq C$. Then the law $\mu$ of the solution of \eqref{problem1.1} satisfies the quadratic transportation cost inequality on $X:= L^2(0,T; L^1(D))$  i.e. $\mu \in T^{2}(C)$.

\end{thm}

\section{Existence and uniqueness of Skeleton Equation} \label{sec:sk}

In this section, we will prove the existence and uniqueness of solution of the skeleton equation \eqref{eq:skeleton}.  To do so, we first construct an
 approximate solution via semi-implicit time discretization scheme, and then derive necessary uniform bounds.

\subsection{Time discretization and a-priori estimate}
 First, we introduce the semi-implicit time discretization. For $N \in \mathbb{N}$, let $0 = t_0 < t_1 < ... < t_N = T$ be a uniform partition of $[0,T]$ with $\tau := \frac{T}{N} = t_{k+1}-t_k$ for all $k=0,1,...,N-1$. For any $u_0 \in L^2(D)$, denote $\bar{u}_{0} = u_{0, \tau}$, where $ u_{0, \tau} \in W_0^{1,p}(D)$ is the unique solution to the problem $u_{0, \tau} - \tau \Delta_{p}u_{0, \tau} = u_0$. Then one has,  (see \cite[Lemma 30]{gv1})
\begin{align} \label{eq:1.1}
\begin{cases}
     u_{0, \tau} \rightarrow u_0 \ in \ L^2(D) \ \ as \ \tau \rightarrow 0, \\
     \frac{1}{2} \left \|u_{0, \tau} \right \|_{L^2(D)}^{2} + \tau \left \| \nabla u_{0, \tau}  \right \|_{L^p(D)}^{p} \leq  \frac{1}{2} \left \|u_{0} \right \|_{L^2(D)}^{2}.
\end{cases}
\end{align}
Let us define a projection  $\Pi_{\tau}$ in time as follows:  for any $h \in  L^2([0,T]; \mathbb{R})$, we define the projection $\Pi_{\tau} h \in \mathcal{P}_{\tau}$ by 
          $$\int_{0}^{T} \left ( \Pi_{\tau} h-h, \psi_{\tau}  \right ) \, dt = 0 \quad \forall \, \psi_{\tau} \in \mathcal{P}_{\tau}\,,$$
          where $\mathcal{P}_{\tau} := \left \{ \psi_{\tau}: (0,T) \rightarrow \mathbf{R} : \psi_{\tau} \mid_{(t_j,t_{j+1}]} \text{is a constant in} \, \mathbb{R}  \right \}.$
          Then it satisfies the following:
          \begin{align} \label{est:SI1}
              \begin{cases}
                  \Pi_{\tau} h \rightarrow h \quad in \, L^2([0,T]; \mathbb{R}),\\
                  \left \|\Pi_{\tau} h \right \|_{L^2([0,T]; \mathbb{R})} \leq \left \| h \right \|_{L^2([0,T]; \mathbb{R})}, 
               \displaystyle 
              \end{cases}
          \end{align}
          For more details refer to see \cite[Lemma 4.4]{Xin}. We  set 
          \begin{align}
          h_{k+1} = \Pi_{\tau} h(t_{k+1}) \quad k=0,1,...,N-1. \label{defi: discrete-h}
          \end{align}
       Then, by \eqref{est:SI1}, we have 
\begin{align}
\tau \sum_{k=0}^{N-1} |h_{k+1}|^2 \le \int_0^T |\Pi_\tau h|^2\,dt \le \int_0^T h^2(t)\,dt \le C\,. \label{esti:proj-time}
\end{align}
Consider a semi-implicit Euler Maruyama scheme of \eqref{eq:skeleton}:  for $k=0,1,...,N-1$,
          \begin{equation}\label{SI1}
              \bar{u}_{k+1} - \bar{u}_{k} - \tau {\rm div}( |\nabla {\bar{u}}_{k+1}|^{p-2}\nabla {\bar{u}}_{k+1} +\vec{f}(\bar{u}_{k+1})) = \tau H(\bar{u}_{k})h_{k+1}\,,
          \end{equation}
          where $h_{k+1}$ is given in \eqref{defi: discrete-h}.  The following proposition ensures the existence of solution of  the semi-implicit scheme \eqref{SI1} 
\begin{prop} \label{extdis}
For any given $\bar{u}_{k} \in L^2(D)$ and any $k = 0,1,2,...,N-1$, there exists a unique $\bar{u}_{k+1} \in W_{0}^{1,p}(D) $ such that
 \[\bar{u}_{k+1} - \bar{u}_{k} - \tau {\rm div}( |\nabla \bar{u}_{k+1}|^{p-2}\nabla \bar{u}_{k+1} + \vec{f}(\bar{u}_{k+1})) = \tau H(\bar{u}_{k})h_{k+1} \] in $L^2(D).$
 \begin{proof}
 For a fixed $\tau > 0$, we define an operator $\mathcal{T} : W_{0}^{1,p}(D) \rightarrow W^{-1,p'}(D)$ by
 $$ \left \langle \mathcal{T}(u), v  \right \rangle_{W^{-1,p'}(D),W_{0}^{1,p}(D)} := \int_{D} \left (u v + \tau( |\nabla u|^{p-2}\nabla u +\vec{f}(u)) \cdot \nabla{v}  \right ) \, dx, \quad \forall \, u,v \in  W_{0}^{1,p}(D). $$
 Then, $\mathcal{T}$ is coercive and pseudo-monotone operator. By employing  Brezis' Theorem \cite [Theorem 2.6]{roubi}, one can easily conclude that $\mathcal{T}$ is onto. Moreover, using the similar arguments as in \cite[Lemma 1]{gv1}, we deduce that $\mathcal{T}$ is injective and $\mathcal{T}^{-1}$ is continuous.
 \vspace{0.1cm}
 
\noindent  Observe that, $\bar{u}_{k} + \tau H(\bar{u}_{k})h_{k+1} \in L^2(D)$. Hence, there exists unique $\bar{u}_{k+1} \in W_{0}^{1,p}(D)$ such that
 \begin{align} \label{S1}
  & \bar{u}_{k+1} = \mathcal{T}^{-1}(\bar{u}_{k} + \tau H(\bar{u}_{k})h_{k+1}) \quad \forall \,k = 0,1,...,N-1.  \notag \\
    & i.e., \ \bar{u}_{k+1}  - \bar{u}_{k} - \tau {\rm div}( |\nabla \bar{u}_{k+1}|^{p-2}\nabla \bar{u}_{k+1} + \vec{f}(\bar{u}_{k+1})) = \tau H(\bar{u}_{k})h_{k+1} \quad \forall \,k = 0,1,...,N-1. \notag
 \end{align}
 Hence the proposition \ref{extdis} follows by induction.
 \end{proof}
 \end{prop}
 Now, we define some functions on the whole time interval $[0,T].$
  \begin{defi}
          For $N \in \mathbb{N}$, $\tau > 0$ define the right-continuous step functions
          \[ u_{\tau}(t) = \sum_{k=0}^{N-1} \bar{u}_{k+1} \chi_{[t_k,t_{k+1})} (t), \quad t \in [0,T], \]
          \[ h_{\tau}(t) = \sum_{k=0}^{N-1} {h}_{k+1} \chi_{[t_k,t_{k+1})} (t), \quad t \in [0,T], \]
          the left-continuous step function
          \[\hat{u}_{\tau}(t) = \sum_{k=0}^{N-1} \bar{u}_{k} \chi_{(t_k,t_{k+1}]} (t), \quad t \in [0,T], \]
          and the piecewise affine function
          \[ \tilde{u}_{\tau}(t) = \sum_{k=0}^{N-1} \left (\frac{\bar{u}_{k+1}-\bar{u}_{k}}{\tau} (t-t_k) + \bar{u}_{k}  \right ) \chi_{[t_k,t_{k+1})} (t), \quad t \in [0,T), \tilde{u}_{\tau}(T) =\bar{u}_N. \]
         \end{defi}
 We wish to derive a-priori estimates for $\{\bar{u}_k\}, u_\tau$ and $\tilde{u}_\tau$.  Regarding this, we have the following lemma.
         \begin{lem}
          There exists a constant $L > 0$, independent of $\tau$, such that
\[\underset{n=0,...,N}\max   \left \| \bar{u}_{n} \right \|^2_{L^2(D)} \leq L, \int_{0}^{T} \left \|{u}_{\tau}(t) \right \|^2_{L^2(D)} \,dt \leq L, \]
\begin{align} \label{bdd0}
      \underset{t \in [0,T]}\sup\left \| u_{\tau}(t) \right \|^2_{L^2(D)}   = \underset{t \in [0,T]}\sup \left \| \tilde{u}_{\tau}(t) \right \|^2_{L^2(D)} \leq \underset{n=0,...,N}\max   \left \| \bar{u}_{n} \right \|^2_{L^2(D)}  \leq L. 
\end{align}
 Moreover, 
 \begin{align} \label{bdd001}
     \int_{0}^{T}\int_{D} \left | \nabla{u_{\tau}} \right |^p dx \, dt  &  \leq L , \quad   \int_{0}^{T}\int_{D} \left | \nabla{\tilde{u}_{\tau}} \right |^p dx \, dt  \leq L,  \notag \\
& \left \|u_{\tau} - \tilde{u}_{\tau} \right \|^2_{L^2(D_{(0,T)})} \leq L \tau .
  \end{align}
Furthermore, 
\begin{align} \label{bdd1}
    \left \| \tilde{u}_{\tau}(t) \right \|^p_{L^p([0,T];W_0^{1,p}(D))} \leq  L.
\end{align}
    \end{lem}
    
  \begin{proof}
     We take a test function $v=\bar{u}_{k+1}$ in  \eqref{SI1} to obtain
   \begin{align*}
    & (\bar{u}_{k+1} - \bar{u}_{k}, \bar{u}_{k+1}) -
    \tau \left \langle {\rm div}( |\nabla {\bar{u}}_{k+1}|^{p-2}\nabla {\bar{u}}_{k+1} +\vec{f}(\bar{u}_{k+1})), \bar{u}_{k+1} \right \rangle = \tau ( H(\bar{u}_{k})h_{k+1},\bar{u}_{k+1})\,.
   \end{align*}
   Using the Young's inequality and the identity $ (x -y) x = \frac{1}{2}\left [x^2 - y^2 + {(x-y)}^2 \right ] \forall \, x,y \in \mathbb{R}$ 
   along with the fact that $\displaystyle \int_{D} \vec{f}(v). \nabla v \, dx = 0 \,~ \forall \, v \in W_{0}^{1,p}(D)$, we arrive at
   
   \begin{align}
   \label{eq:estimation}
   \frac{1}{2}\left [\left \| \bar{u}_{k+1} \right \|_{L^2(D)}^{2} - \left \| \bar{u}_{k} \right \|_{L^2(D)}^{2} + \left \|\bar{u}_{k+1}-\bar{u}_{k}  \right \|_{L^2(D)}^{2} \right ]  + \tau \left \| \nabla \bar{u}_{k+1}  \right \|_{L^p(D)}^{p} \notag \\
   \leq \tau \left \| H(\bar{u}_{k}) \right \|_{L^2(D)} |h_{k+1}| \left \| \bar{u}_{k+1} \right \|_{L^2(D)}.
   \end{align}
 \vspace{0.05cm}
   
  Discarding nonnegative term and applying Young's inequality together with \ref{A3}, we have
  
   \begin{align}
         \frac{1}{2}\left [\left \| \bar{u}_{k+1} \right \|_{L^2(D)}^{2} - \left \| \bar{u}_{k} \right \|_{L^2(D)}^{2} \right ] & \leq \tau \left [\frac{1}{2 \delta} \left \| \bar{u}_{k+1} \right \|_{L^2(D)}^{2} + \frac{\delta}{2} \left \| H(\bar{u}_{k}) \right \|_{L^2(D)}^{2} |h_{k+1}|^{2} \right ] \notag \\
         & \leq   \frac{1}{4} \left \| \bar{u}_{k+1} \right \|_{L^2(D)}^{2} + C {\tau}^{2} \left \| \bar{u}_{k} \right \|_{L^2(D)}^{2} |h_{k+1}|^{2} \quad (for \ \delta = \frac{\tau}{2}). \notag 
          \end{align}
        Therefore, we get
         \begin{align}
        \frac{1}{4}\left [\left \| \bar{u}_{k+1} \right \|_{L^2(D)}^{2} - \left \| \bar{u}_{k} \right \|_{L^2(D)}^{2} \right ]  \leq C {\tau}^{2} \left \| \bar{u}_{k} \right \|_{L^2(D)}^{2} |h_{k+1}|^{2}.
   \end{align}
   For fixed $n \in \left \{1,2,...,N  \right \}$, we take the sum over $k=0,1,...,n-1$ in above equation and obtain
   \[ \left \| \bar{u}_{n} \right \|_{L^2(D)}^{2} - \left \| \bar{u}_{0} \right \|_{L^2(D)}^{2} \leq C {\tau}^{2} \sum_{k=0}^{n-1}  \left \| \bar{u}_{k} \right \|_{L^2(D)}^{2} |h_{k+1}|^{2}. \]
   Using the discrete Gr{\"o}nwall's inequality along with \eqref{esti:proj-time}, we conclude that
   \begin{align}
   \label{eq:estimation1}
   \sup_{0\le n\le N}     \left \| \bar{u}_{n} \right \|_{L^2(D)}^{2} \leq L\,,
   \end{align}
   for some $L>0$, independent of $\tau$. By definition of $u_{\tau}$, $\tilde{u}_{\tau}$ and \eqref{eq:estimation1}, one easily get the estimation of \eqref{bdd0}. 
   \vspace{0.1cm}

   To get \eqref{bdd001}, we proceed as follows.  Taking the sum over $0,1,...,N-1$ in \eqref{eq:estimation}, discarding nonnegative term and using \eqref{eq:estimation1}, \eqref{esti:proj-time}, we have
    \begin{align*}
        \int_{0}^{T} \left \| \nabla {u}_{\tau}  \right \|_{L^p(D)}^{p} dt & = \sum_{k=0}^{N-1} \int_{t_k}^{t_{k+1}} \left \| \nabla \bar{u}_{k+1}  \right \|_{L^p(D)}^{p} dt = \tau \sum_{k=0}^{N-1}  \left \| \nabla \bar{u}_{k+1}  \right \|_{L^p(D)}^{p} \\
       & \leq \left \| \bar{u}_{0} \right \|_{L^2(D)}^{2} + C {\tau}^{2} \sum_{k=0}^{N-1}  \left \| \bar{u}_{k} \right \|_{L^2(D)}^{2} |h_{k+1}|^{2} 
        \leq  C {\tau}^{2} \sum_{k=0}^{N-1} |h_{k+1}|^{2} \leq  L.
    \end{align*} 
    Similar estimate holds for $\tilde{u}_\tau$.  Now, by definition of $u_{\tau}$, $\tilde{u}_{\tau}$ and the  fact that 
        $$ \sum_{k=0}^{N-1}  \left \| \bar{u}_{k+1}-\bar{u}_{k} \right \|^2_{L^2(D)}\le C,$$
     which could be derived from \eqref{eq:estimation} together with  \eqref{eq:estimation1}, we have
         \begin{align}
        \int_{0}^{T} \left \|u_{\tau} - \tilde{u}_{\tau} \right \|^2_{L^2(D)} \, dt & = \sum_{k=0}^{N-1} \int_{t_k}^{t_{k+1}}  \left \| \left (1-\frac{t-t_k}{\tau} \right )  (\bar{u}_{k+1}-\bar{u}_{k}) \right \|^2_{L^2(D)} \, dt \notag \\
        & = \sum_{k=0}^{N-1} \int_{t_k}^{t_{k+1}}  \left \| \frac{t_{k+1} - t}{\tau}  (\bar{u}_{k+1}-\bar{u}_{k}) \right \|^2_{L^2(D)} \, dt \notag \\
        & = \frac{\tau}{3} \sum_{k=0}^{N-1}  \left \| \bar{u}_{k+1}-\bar{u}_{k} \right \|^2_{L^2(D)} \leq L \tau. \notag
    \end{align}
 Note that 
 \begin{align*}
   \left \| \tilde{u}_{\tau}(t) \right \|^p_{L^p([0,T];W_0^{1,p}(D))} \le C \Big(  \int_{0}^{T} \left \| \nabla {u}_{\tau}  \right \|_{L^p(D)}^{p}\,dt + \tau \|\nabla \bar{u}_0\|_{L^p(D)}^p \Big)\,.
 \end{align*}
 Hence the assertion \eqref{bdd1} follows from the above estimate once we use \eqref{bdd001} and \eqref{eq:1.1}.  This completes the proof. 
  \end{proof}
\subsection{Convergence analysis of $\{\tilde{u}_\tau\}$} This subsection is devoted to analyze the convergence of the family $\{\tilde{u}_\tau\}$ in some appropriate space. Before that, we need some preparation. Let $(\mathbb{X}, \left \| \cdot \right \|_{\mathbb{X}})$ be a separable metric space. For any $q > 1$ and $\alpha \in (0,1)$, Let $W^{\alpha,q}([0,T];\mathbb{X})$ be the Sobolev space  of all $u \in L^{q}([0,T];\mathbb{X})$ such that
\begin{align} \label{norm def1}
   \int_{0}^{T} \int_{0}^{T} \frac{\left \| u(t) - u(s) \right \|_{\mathbb{X}}^{q}}{\left | t-s \right |^{1+\alpha q}} \, dt \, ds < \infty 
\end{align}
with the norm 
\[ \left \| u \right \|_{W^{\alpha,q}([0,T];\mathbb{X})}^{q} = \int_{0}^{T}  \left \| u(t) \right \|_{\mathbb{X}}^{q} \, dt + \int_{0}^{T} \int_{0}^{T} \frac{\left \| u(t) - u(s) \right \|_{\mathbb{X}}^{q}}{\left | t-s \right |^{1+\alpha q}} \, dt \, ds.  \]

\begin{lem}
\label{lem:cpt0}
The following estimation holds: for $\alpha \in (0, \frac{1}{p})$
\[ \sup_{\tau > 0} \left\{ \left\| \tilde{u}_{\tau}\right\|_{W^{\alpha,p}([0,T];{W^{-1,p'}(D)})}  \right\}  < \infty. \]
\end{lem}
\begin{proof}
By Sobolev embedding $W_{0}^{1,p}(D) \hookrightarrow W^{-1,p'}(D)$, we get
\begin{align}
 \int_{0}^{T} \left \|\tilde{u}_{\tau}  \right \|_{W^{-1,p'}(D)}^{p} \, dt  \leq C \int_{0}^{T} \left \|\tilde{u}_{\tau} \right \|_{W_{0}^{1,p}(D)}^{p} \, dt  \leq C. \label{esti:1-frac-sov}
 \end{align}
We rewrite \eqref{SI1} in terms of ${u}_{\tau}, {h}_{\tau}, \hat{u}_{\tau}$ and $\tilde{u}_{\tau}$ as
\begin{align} \label{eq:cpt000}
    \tilde{u}_{\tau}(t) &= u_{0, \tau} + \int_{0}^{t} {\rm div}_x ( |\nabla {u}_{\tau}|^{p-2}\nabla {u}_{\tau}) \, ds
    +  \int_{0}^{t} {\rm div}_x \vec{f}({u}_{\tau}) \, ds + \int_{0}^{t} H(\hat{u}_{\tau}){h}_{\tau}(s) \, ds \notag\\
    & \equiv \mathcal{K}_{0}^\tau(t) + \mathcal{K}_{1}^{\tau}(t) + \mathcal{K}_{2}^{\tau}(t) + \mathcal{K}_{3}^{\tau}(t)\,.
\end{align}
In view  of \eqref{esti:1-frac-sov}, it remains to show  that $\mathcal{K}_{i}^{\tau}(t)$ satisfies \eqref{norm def1} for $0 \leq i \leq 3$ and $\alpha \in (0,\frac{1}{p})$.
Since $\mathcal{K}_{0}^{\tau}(t)$ is independent of time, clearly it satisfies \eqref{norm def1} for any  $\alpha \in (0,1)$.
W.L.O.G., we assume that $s < t$. By employing Jensen's inequality and \cite[Equation 4.1.14]{claud}, we have 

\begin{align}
    \left \|\mathcal{K}_{1}^{\tau}(t) - \mathcal{K}_{1}^{\tau}(s)   \right \|_{W^{-1,p'}(D)}^{p}  & \leq \left \| \int_{s}^{t} \Delta_p {u}_{\tau}  \, dr  \right \|_{W^{-1,p'}(D)}^{p} \leq \bigg( \int_{s}^{t} \left \| \Delta_p {u}_{\tau}  \right \|_{W^{-1,p'}(D)}  \, dr  \bigg)^{p}  \notag \\
     & \leq \bigg(\int_{s}^{t} \left \|{u}_{\tau} \right \|_{W_{0}^{1,p}(D)}^{p-1}  \, dr \bigg)^{p} \leq C (t-s) \bigg( \int_{0}^{T} \left \|{u}_{\tau} \right \|_{W_{0}^{1,p}(D)}^{p} \, dr  \bigg)^{p-1}. \notag
   \end{align}
  Thus, there exists a constant $C>0$, independent of $\tau$, such that for any $\alpha \in (0, \frac{1}{p})$, 
\begin{align} \label{eq:cpt01}
   \int_{0}^{T}\int_{0}^{T} \frac{\left \|\mathcal{K}_{1}^{\tau}(t) - \mathcal{K}_{1}^{\tau}(s)   \right \|_{W^{-1,p'}(D)}^{p}}{\left | t-s \right |^{1+\alpha p}} \, dt \, ds \leq C.  
\end{align}
We use the assumption \ref{A2} to get
\begin{align*}
    \left \|\mathcal{K}_{2}^{\tau}(t) - \mathcal{K}_{2}^{\tau}(s)   \right \|_{W^{-1,p'}(D)}^{p} &\leq \left \| \int_{s}^{t}  {\rm div}_x \vec{f}({u}_{\tau}) \, dr   \right \|_{W^{-1,p'}(D)}^{p} \leq  C \bigg( \int_{s}^{t} \left \|  {u}_{\tau}  \right \|_{L^{2}(D)} \, dr \bigg)^{p}  \notag \\
    & \leq C (t-s)^{p-1} \int_{0}^{T} \left \| {u}_{\tau} \right \|_{W_0^{1,p}(D)}^{p} \, dr,
\end{align*}
and hence  for any $\alpha \in (0,\frac{1}{p})$
\begin{align} \label{eq:cpt02}
    \int_{0}^{T}\int_{0}^{T} \frac{\left \|\mathcal{K}_{2}^{\tau}(t) - \mathcal{K}_{2}^{\tau}(s)   \right \|_{W^{-1,p'}(D)}^{p}}{\left | t-s \right |^{1+\alpha p}} \, dt \, ds \leq C.
\end{align}

Similarly,  by employing assumption \ref{A3}, one has
\begin{align}
    \left \|\mathcal{K}_{3}^{\tau}(t) - \mathcal{K}_{3}^{\tau}(s)   \right \|_{L^{2}(D)}^{p} &\leq \left \| \int_{s}^{t}  H({u}_{\tau}){h}_{\tau}(r) \, dr  \right \|_{L^{2}(D)}^{p} \leq C \bigg( (t-s) \int_{s}^{t}  \left \| {u}_{\tau} \right \|_{L^2(D)}^{2} \left | {h}_{\tau}(r) \right |^{2} \, dr \bigg)^{\frac{p}{2}}  \notag \\
    &\leq  C \bigg( (t-s) \sup_{0\leq t\leq T}\left \| {u}_{\tau} \right \|_{L^2(D)}^{2} \int_{s}^{t} \left | {h}_{\tau}(r) \right |^{2} \, dr  \bigg)^{\frac{p}{2}} \leq C (t-s)^{\frac{p}{2}} \bigg( \int_{0}^{T} \left | {h}_{\tau}(r) \right |^{2} \, dr  \bigg)^{\frac{p}{2}}.  \notag
\end{align}

Therefore, by \eqref{esti:proj-time},  we have
\begin{align} \label{eq:cpt03}
   \displaystyle \int_{0}^{T}\int_{0}^{T} \frac{\left \|\mathcal{K}_{3}^{\tau}(t) - \mathcal{K}_{3}^{\tau}(s)   \right \|_{W^{-1,p'}(D)}^{p}}{\left | t-s \right |^{1+\alpha p}} \, dt \, ds \leq C, \quad \forall \, \alpha \in (0,\frac{1}{p}).
\end{align}
Putting the inequalities \eqref{eq:cpt01}-\eqref{eq:cpt03} in \eqref{eq:cpt000}, we arrive at the assertion that
$$ \underset {\tau > 0}\sup \left\| \tilde{u}_{\tau}\right\|_{W^{\alpha,p}([0,T];{W^{-1,p'}(D)})} \leq C, \quad \forall \, \alpha \in (0,\frac{1}{p}), $$
where $C>0$ is a constant, independent of $\tau$.
\end{proof}
As mentioned earlier, we would like to have strong convergence of $\{\tilde{u}_\tau\}$ in some appropriate space. For that, we will use the following well-known lemma.
  \begin{lem} \cite[Theorem 2.1]{fland}
\label{lem:cpt}
Let $\mathbb{X} \subset \mathbb{Y} \subset \mathbb{X^*}$ be Banach spaces, $\mathbb{X}$ and $\mathbb{X^*}$ reflexive, with \\
compact embedding of $\mathbb{X}$ in $\mathbb{Y}$. For any $q \in (1,\infty)$ and $\alpha \in (0,1)$, the  embedding of \\ $L^{q}([0,T];\mathbb{X}) \cap W^{\alpha,q}([0,T];\mathbb{X^*}) $  equipped with natural norm in $L^{q}([0,T];\mathbb{Y})$ is compact.
\end{lem}

Let $ \mathcal{O} = L^{2}([0,T];L^{2}(D)) \cap C([0,T];W^{-1,p'}(D)).$
\begin{thm}
\label{eq:cgs1}
There exists a subsequence of $\{\tilde{u}_\tau\}$, still denoted by $\{\tilde{u}_{\tau}\}$ and \\
$v \in L^{2}([0,T];L^{2}(D)) \, \cap \, L^{p}([0,T];W_{0}^{1,p}(D)) \, \cap \, L^{\infty}([0,T];L^{2}(D))$ such that
\begin{itemize}
    \item[(a)] $\tilde{u}_{\tau} \rightarrow v$ in $\mathcal{O},$
\item[(b)]  $\tilde{u}_{\tau} \rightharpoonup v $ in $L^p([0,T];W_0^{1,p}(D)),$
\item[(c)] $\tilde{u}_{\tau} \overset{*}{\rightharpoonup} v $ in $L^{\infty}([0,T];L^{2}(D)).$
\end{itemize}
\end{thm}
\begin{proof}
\text{Proof of (a).}  
It is easy to view from  Lemma \ref{lem:cpt}  that $ L^{p}([0,T];W_{0}^{1,p}(D)) \, \cap \, W^{\alpha,p}([0,T];{W^{-1,p'}(D)})$ is compactly embedded in $L^{2}([0,T];L^{2}(D))$. Thanks to Lemmas \ref{lem:cpt} and \ref{lem:cpt0}, we see that $\left \{ \tilde{u}_{\tau} \right \}$ is precompact in $L^{2}([0,T];L^{2}(D))$. In view of the proof of Lemma \ref{lem:cpt0}, one get
\begin{align*}
\|\tilde{u}_\tau(t)-\tilde{u}_\tau(s)\|_{W^{-1,p'}(D)} \le C |t-s|^\beta 
\end{align*}
for some $\beta>0$. Moreover, by the compact embedding of $L^{2}(D) \hookrightarrow W^{-1,p'}(D)$ together with \eqref{bdd0}, we see that the family $\{\tilde{u}_\tau\}$ is uniformly bounded in $W^{-1,p'}(D)$. Hence  Arzel{\'a}–Ascoli theorem yields that the family  $\left \{ \tilde{u}_{\tau} \right \}$ is pre-compact in $C([0,T];W^{-1,p'}(D))$. Hence the assertion follows. 
\vspace{0.2cm}

\noindent{ Proof of ${\rm (b)}$.} We use the estimate \eqref{bdd1} and $p > 2$ to obtain that, there exists a subsequence of $\{\tilde{u}_\tau\}$, still denoted by $\{\tilde{u}_{\tau}\}$, and $Z \in L^p([0,T];W_0^{1,p}(D))$  such that $\tilde{u}_{\tau}\rightharpoonup Z $ in $L^{2}([0,T];W_0^{1,2}(D))$
i.e. $$\int_{0}^{T} \int_{D} \tilde{u}_{\tau}(t,x) \psi(t,x) \, dx \, dt  \rightarrow \int_{0}^{T} \int_{D} Z(t,x) \psi(t,x) \, dx \, dt  \quad \forall \, \psi \in L^{2}([0,T];W^{-1,2}(D)). $$
Note that $L^{2}(D) \hookrightarrow W^{-1,2}(D)$, so in particular $\tilde{u}_{\tau} \rightharpoonup Z $ in $L^{2}([0,T];L^{2}(D))$. By part $(a)$, we have seen that $\tilde{u}_{\tau}(t) \rightharpoonup v $ in $L^{2}([0,T];L^{2}(D))$, and hence  the assertion follows thanks to the uniqueness of weak limit.
\vspace{0.2cm}

\noindent{Proof of ${\rm (c)}$.} One can follow the similar arguments (under cosmetic changes) as in the proof of \cite[Lemma 17(6)]{gv1}, see also \cite[Lemma 3.8(iii)]{maj} to conclude the result ${\rm (c)}$. 
\end{proof}

\begin{rem} \label{rem1}
Results of Theorem \ref{eq:cgs1} remains valid if one replaces $\tilde{u}_{\tau}$ by $u_{\tau}$ or $\hat{u}_{\tau}$.  
\end{rem}

\begin{lem}
\label{eq:cgseq1}
For any $\phi \in W_0^{1,p}(D)$, the followings hold:
\begin{itemize}
    \item[(i)] $\displaystyle \underset{\tau \rightarrow 0}\lim \int_{0}^{T} \left<  {\rm div}_x \big(\vec{f}({u}_{\tau}(t)) -\vec{f}({v(t)})\big), \phi \right> \, dt = 0.$
    \item[(ii)]  $\displaystyle \underset{\tau \rightarrow 0}\lim \int_{0}^{T} \left<H(\hat{u}_{\tau}(t)){h}_{\tau}(t) - H(v(t))h(t) , \phi \right> \, dt = 0.$
    \item[(iii)] There exists $ F \in L^{p'}([0,T];L^{p'}(D))^{d}$ such that
    
  $\displaystyle \underset{\tau \rightarrow 0}\lim \int_{0}^{T} \left<  {\rm div}_x (|\nabla {u}_{\tau}(t)|^{p-2}\nabla {u}_{\tau}(t) - F(t)), \phi \right> \, dt = 0.$
\end{itemize}
Moreover, the limit function $v$ satisfies the weak form 
\begin{align}
 \displaystyle \left<v(t),\phi \right> = \left<u_{0}, \phi \right> + \int_{0}^{t} \left<  {\rm div}_x (F(s)+\vec{f}({v(s)})), \phi \right> \, ds + \int_{0}^{t} \left< H(v(s))h(s) , \phi \right> \, ds. \label{eq:weak-form-limit-function-skeleton}
 \end{align}
\end{lem}
\begin{proof}
\noindent{Proof of (i):} With the help of assumption \ref{A2}, integration by parts formula, and Remark \ref{rem1} we get, remembering $p>2$,
\begin{align*}
   &  \left | \int_{0}^{T} \left<  {\rm div}_x (\vec{f}({u}_{\tau}(t)) -\vec{f}({v(t)})), \phi \right> \, dt \right |  \leq \int_{0}^{T} \left | - \left<   \vec{f}({u}_{\tau}(t)) -\vec{f}({v(t)}), \nabla{\phi} \right> \right | \, dt \\
    & \leq \int_{0}^{T} ||\vec{f}({u}_{\tau}(t) -\vec{f}({v(t)})||_{L^{p'}(D)} ||\nabla{\phi}||_{L^{p}(D)} \, dt   \leq ||\nabla{\phi}||_{L^{p}(D)} \int_{0}^{T} ||\vec{f}({u}_{\tau}(t) -\vec{f}({v(t)})||_{L^{2}(D)} \, dt  \\
    & \leq C ||\phi||_{W_0^{1,p}(D)}  \int_{0}^{T} ||{u}_{\tau}(t) -{v(t)}||_{L^{2}(D)} \, dt  \leq C ||\phi||_{W_0^{1,p}(D)} ||{u}_{\tau}(t) -{v(t)}||_{L^2([0,T];L^2(D))} \rightarrow 0.
\end{align*}

Proof of (ii): Using $h_{\tau} \rightarrow h$, Theorem \ref{eq:cgs1} and  the assumption \ref{A3}, it is easy to see that
\begin{align*}
    & \left | \int_{0}^{T} \left<H(\hat{u}_{\tau}(t)){h}_{\tau}(t) - H(v(t))h(t) , \phi \right> \, dt \right | \\
     &\leq \left | \int_{0}^{T} \left< \big(H(\hat{u}_{\tau}(t)) - H(v(t)) \big) {h}_{\tau}(t) , \phi \right> \, dt + \int_{0}^{T} \left< \big(h_{\tau}(t) - h(t) \big)H(v(t)) , \phi \right> \, dt  \right | \\
     & \leq \int_{0}^{T} \left|\left< \big(H(\hat{u}_{\tau}(t)) - H(v(t)) \big) {h}_{\tau}(t) , \phi \right> \right| \, dt + \int_{0}^{T} \left| \left< \big(h_{\tau}(t) - h(t) \big)H(v(t)) , \phi \right> \right| \, dt \\
     & \leq \int_{0}^{T} C  ||\hat{u}_{\tau}(t)-v(t)||_{L^2(D)} |h_{\tau}(t)| \, ||\phi||_{W_0^{1,p}(D)} \, dt + \int_{0}^{T} C |h_{\tau}(t)-h(t)| \,  ||v(t)||_{L^2(D)}  ||\phi||_{W_0^{1,p}(D)} \, dt \\
     & \leq C \bigg(\int_{0}^{T} ||\hat{u}_{\tau}(t)-v(t)||_{L^2(D)}^{2} \, dt \bigg)^{\frac{1}{2}} \bigg(\int_{0}^{T} |h_{\tau}(t)|^2 dt \bigg)^{\frac{1}{2}} \notag \\
     & \hspace{2cm} + C \bigg(\int_{0}^{T} |h_{\tau}(t)-h(t)|^2 dt \bigg)^{\frac{1}{2}} \bigg(\int_{0}^{T} ||v(t)||_{L^2(D)}^2 \, dt \bigg)^{\frac{1}{2}} \\
     & \leq C \bigg(\int_{0}^{T} ||\hat{u}_{\tau}(t)-v(t)||_{L^2(D)}^{2} \, dt \bigg)^{\frac{1}{2}} + C \bigg(\int_{0}^{T} |h_{\tau}(t)-h(t)|^2 dt \bigg)^{\frac{1}{2}} \rightarrow 0.
\end{align*}
In other words, we get
\[\displaystyle \underset{\tau \rightarrow 0}\lim \int_{0}^{T} \left<H(\hat{u}_{\tau}(t)){h}_{\tau}(t) - H(v(t)h(t) , \phi \right> \, dt = 0.\]

\noindent{Proof of ${\rm (iii)}$}: By Theorem \ref{eq:cgs1} (see also Remark \ref{rem1}), there exists a subsequence such that $ \nabla {u}_{\tau} \rightharpoonup \nabla v $ in $ L^{p}([0,T];L^{p}(D))^{d}$  as $\tau \rightarrow 0$. Since $\left | |\nabla {u}_{\tau}|^{p-2}\nabla {u}_{\tau} \right |^{p'} = |\nabla {u}_{\tau}|^p,$ there exists $F$ such that $| \nabla {u}_{\tau}|^{p-2}\nabla {u}_{\tau} \rightharpoonup F $ in $ L^{p'}([0,T];L^{p'}(D))^{d}$ as $\tau \rightarrow 0$. Clearly, the following holds:
\[\displaystyle \underset{\tau \rightarrow 0}\lim \int_{0}^{T} \left<  {\rm div}_x (|\nabla {u}_{\tau}(t)|^{p-2}\nabla {u}_{\tau}(t) - F(t)), \phi \right> \, dt = 0.\]
In view of the convergence results in ${\rm (i)}$-${\rm (iii)}$, part ${\rm (a)}$ of Theorem \ref{eq:cgs1} and the first part of \eqref{eq:1.1}, one can pass to the limit in 
\eqref{eq:cpt000} to conclude easily that the limit function $v$ satisfies the weak form \eqref{eq:weak-form-limit-function-skeleton}. This completes the proof.
\end{proof}
\subsection{Existence of solution of \eqref{eq:skeleton}:} \label{subsec:existence-skeleton}
 We show that the limit function $v$ is indeed a solution of \eqref{eq:skeleton}. In view of the weak form 
\eqref{eq:weak-form-limit-function-skeleton}, it it is enough to prove that $F=|\nabla{v}|^{p-2}\nabla{v}$.  For this identification, we proceed as follows.
Take a test function $\bar{u}_{k+1}$ in \eqref{SI1} and  use the identity $ (x -y) x = \frac{1}{2}\left [x^2 - y^2 + {(x-y)}^2 \right ] \forall \, x,y \in \mathbb{R}$. The result is 
\begin{align*}
 & \frac{1}{2}[||\bar{u}_{k+1}||_{L^2(D)}^{2} - ||\bar{u}_{k}||_{L^2(D)}^{2} + ||\bar{u}_{k+1} - \bar{u}_{k}||_{L^2(D)}^{2}] +  \tau \int_{D} |\nabla {\bar{u}}_{k+1}|^{p-2}\nabla {\bar{u}}_{k+1} \cdot \nabla \bar{u}_{k+1} \, dx \\
   & \hspace{2cm} = \tau \int_{D} H(\bar{u}_{k})h_{k+1} \cdot \bar{u}_{k+1} \, dx.
\end{align*}
Summing over $k = 0,1,...,N-1$ and using the definition that $\tilde{u}_{\tau}(T) =\bar{u}_N$, we obtain
\begin{align}
     \frac{1}{2}||\tilde{u}_{\tau}(T)||_{L^2(D)}^{2} +  \int_{0}^{T} \int_{D}  |\nabla u_{\tau}(t)|^{p-2}\nabla u_{\tau}(t) \cdot \nabla u_{\tau}(t) \, dx \, dt \notag \\
     -  \int_{0}^{T} \int_{D}  H(\hat{u}_{\tau})(t)h_{\tau}(t) \cdot {u}_{\tau}(t) \, dx \, dt \leq \frac{1}{2}||u_{0, \tau}||_{L^2(D)}^{2}. \label{eq:discrte-ito-type}
\end{align}
On the other hand, from Lemma \ref{eq:cgseq1}, we get
\begin{align}
  \left \|v(T) \right \|_{L^2(D)}^{2}  + 2 \int_{0}^{T} \int_{D}   F \cdot \nabla {v(t)} \, dx \, dt  = \left \|u_{0} \right \|_{L^2(D)}^{2} + 2 \int_{0}^{T} \left<H(v(t))h(t),v(t) \right> \, dt .  \label{eq:ito-type}
\end{align}
Subtracting \eqref{eq:ito-type} from \eqref{eq:discrte-ito-type}, one has
\begin{align*}
   &  \frac{1}{2}  \big[||\tilde{u}_{\tau}(T)||_{L^2(D)}^{2}  - \left \|v(T) \right \|_{L^2(D)}^{2}\big] + \Big\{ \int_{0}^{T} \int_{D}  |\nabla u_{\tau}(t)|^{p-2}\nabla u_{\tau}(t) \cdot \nabla u_{\tau}(t) \, dx \, dt - \int_{0}^{T} \int_{D}   F \cdot \nabla {v(t)} \, dx \, dt \Big\} \\
     & \leq  \frac{1}{2} \big[||u_{0, \tau}||_{L^2(D)}^{2} - \left \|u_{0} \right \|_{L^2(D)}^{2}\big] + \Big\{ \int_{0}^{T}  \left \langle H(\hat{u}_{\tau}(t))h_{\tau}(t) \cdot {u}_{\tau}(t) \right \rangle dt -  \int_{0}^{T} \left<H(v(t))h(t),v(t) \right> \, dt\Big\}.
\end{align*}
It is easy to prove that,
\begin{align*}
    \begin{cases}
   \displaystyle  \underset{\tau}\liminf \, ||\tilde{u}_{\tau}(T)||_{L^2(D)}^{2}  - \left \|v(T) \right \|_{L^2(D)}^{2} \geq 0, \\ \\
    \displaystyle  ||u_{0, \tau}||_{L^2(D)}^{2} \rightarrow \left \|u_{0} \right \|_{L^2(D)}^{2},\\
   \\
    \displaystyle  \int_{0}^{T}  \left \langle H(\hat{u}_{\tau}(t)) h_{\tau}(t), {u}_{\tau}(t) \right \rangle dt \rightarrow  \int_{0}^{T} \left<H(v(t))h(t),v(t) \right> \, dt.
    \end{cases}
\end{align*}
Thus, we obtain
\begin{align}
    \underset{\tau}\limsup \left [ \int_{0}^{T} \int_{D}  |\nabla u_{\tau}(t)|^{p-2}\nabla u_{\tau}(t) \cdot \nabla u_{\tau}(t) \, dx \, dt \right ] \leq  \int_{0}^{T} \int_{D}   F \cdot \nabla {v(t)} \, dx \, dt .
\end{align}
 By using the following inequality
 \begin{align} \label{inequ}
     2^{2-p} |a-b|^{p} \leq \, (|a|^{p-2}a - |b|^{p-2}b) \cdot(a-b), \quad a,b \in \mathbb{R}^d,~~p>2
 \end{align}
  and the fact that ~(by Remark \ref{rem1})
 $$\nabla {u}_{\tau} \rightharpoonup \nabla v ~~\text{in}~~ L^{p}([0,T];L^{p}(D))^{d},\quad |\nabla {u}_{\tau}|^{p-2}\nabla{u}_{\tau} \rightharpoonup F ~~\text{in}~~ L^{p'}([0,T];L^{p'}(D))^{d},$$
 we obtain
 \begin{align*}
     & \underset{\tau \rightarrow 0}\limsup \left [ \int_{0}^{T} \int_{D}  |\nabla {u}_{\tau} - \nabla v|^{p} \, dx \, dt \right ] \\
      & \leq \underset{\tau \rightarrow 0}\limsup \left [ \int_{0}^{T} \int_{D} \left ( |\nabla {u}_{\tau}|^{p-2}\nabla {u}_{\tau}  - |\nabla v|^{p-2}\nabla v \right ) \cdot \nabla ({u}_{\tau} - v) \, dx \, dt \right ] \\
        & \leq \underset{\tau \rightarrow 0}\limsup \left [\int_{0}^{T} \int_{D} |\nabla {u}_{\tau}|^{p-2}\nabla {u}_{\tau} \cdot \nabla {u}_{\tau} \, dx \, dt \right ] - \int_{0}^{T} \int_{D} F \cdot \nabla v \, dx \, dt \leq 0.
 \end{align*}
Therefore, we get $\nabla {u}_{\tau} \rightarrow \nabla v ~~\text{in}~~ L^{p}([0,T];L^{p}(D))^{d}.$ Moreover, since $\left | |\nabla{u}_{\tau}|^{p-2}\nabla {u}_{\tau} \right |^{p'} = |\nabla{u}_{\tau}|^p$, one easily see that  $\left \| |\nabla {u}_{\tau}|^{p-2}\nabla{u}_{\tau} \right \|_{L^{p'}([0,T];L^{p'}(D))^{d}} \rightarrow \left \| |\nabla {v}|^{p-2}\nabla{v} \right \|_{L^{p'}([0,T];L^{p'}(D))^{d}} $. Hence by  Radon–Riesz property \cite{Robert} of the space $L^{p'}([0,T];L^{p'}(D))^{d}$, we arrive that $ |\nabla {u}_{\tau}|^{p-2}\nabla{u}_{\tau} \rightarrow |\nabla {v}|^{p-2}\nabla{v}$ in $L^{p'}([0,T];L^{p'}(D))^{d}$  i.e., $F = |\nabla {v}|^{p-2}\nabla{v}.$  This completes the existence proof. 
 
 
\subsection{Proof of uniqueness} \label{sec:uniqueness}
In this subsection, we prove uniqueness of the solution of \eqref{eq:skeleton} via $L^1$-contraction principle. To do so, let $h \in L^2([0,T]; \mathbb{R})$ be fixed and $u_{1}$ and $u_{2}$ be two solution of \eqref{eq:skeleton}. Let us consider the following convex approximation of the absolute value function. Let $\zeta : \mathbb{R} \rightarrow \mathbb{R}$ be a $C^{\infty}$ function satisfying 

\[ \zeta (0) = 0, \quad \zeta (-r) = \zeta (r), \quad \zeta^{'}(-r) = - \zeta^{'}(r), \quad \zeta^{''} \geq 0, \]
and
\begin{center}
    $ \zeta^{'}(r) = \left\{\begin{matrix}
  -1 & \text{when} \ r \leq 0, \\ 
  \in [-1,1] & \text{when} \ |r| < 1, \\
  1 & \text{when} \ r \geq 1.
\end{matrix}\right.$
\end{center}
 
 For any $\vartheta > 0$, define $\zeta_{\vartheta} : \mathbb{R} \rightarrow \mathbb{R}$ by $\zeta_{\vartheta} = \vartheta \zeta (\frac{r}{\vartheta})$. Then 
 \begin{align}
  |r| - K_1 \vartheta \leq \zeta_{\vartheta} \leq |r|,  \quad |\zeta_{\vartheta}^{''}(r)| \leq \frac{K_2}{\vartheta} \textbf{1}_{|r| \leq \vartheta}, \label{esti:approx}
  \end{align}
 where $K_1 = \underset{|r|\leq 1}\sup \left | |r| - \zeta(r) \right |$ and $K_2 = \underset{|r|\leq 1}\sup |\zeta^{''}(r)|$.
 \vspace{0.5cm}

Using chain-rule and integration by parts formula, we have
 \begin{align}
&\int_{D} \zeta_{\vartheta}(u_1(t) - u_2(t)) dx \notag \\
& =  - \int_{0}^{t}\int_{D} (|\nabla u_1|^{p-2}\nabla u_1 - |\nabla u_2|^{p-2}\nabla u_2 ) . \nabla(u_1 - u_2)(s) \, \zeta_{\vartheta}^{''}(u_1 - u_2) \, dx \, ds  \notag \\
  & \quad - \int_{0}^{t}\int_{D} (\vec{f}(u_1(s,x)) - \vec{f}(u_2(s,x))) \cdot \nabla(u_1 - u_2)(s) \, \zeta_{\vartheta}^{''}(u_1 - u_2) \, dx \, ds \notag \\
  & \qquad  + \int_{0}^{t}\int_{D} \zeta_{\vartheta}^{'}(u_1 - u_2) (H(u_1) - H(u_2)) \, h(s)  \, dx \, ds\,. \notag
 \end{align}
Thanks to the inequality \eqref{inequ} and the fact that $\zeta_{\vartheta}^{''} \geq 0$,  we see that 
 \begin{align*}
 &- (|\nabla u_1|^{p-2}\nabla u_1 - |\nabla u_2|^{p-2}\nabla u_2 ) \cdot \nabla(u_1 - u_2)(s) \, \zeta_{\vartheta}^{''}(u_1 - u_2) \\
 & \leq -C \, |\nabla(u_1 - u_2)|^{p} \, \zeta_{\vartheta}^{''}(u_1 - u_2) \leq 0,
 \end{align*}
and therefore we obtain
 \begin{align}  \label{eq:uni2}
     \int_{D} \zeta_{\vartheta}(u_1(t) - u_2(t)) \, dx  &
     \leq -  \int_{0}^{t}\int_{D} (\vec{f}(u_1(s,x)) - \vec{f}(u_2(s,x))) \cdot \nabla(u_1 - u_2)(s) \zeta_{\vartheta}^{''}(u_1 - u_2) \, dx \, ds  \notag \\ 
    &  \qquad + \int_{0}^{t}\int_{D} \zeta_{\vartheta}^{'}(u_1 - u_2) (H(u_1) - H(u_2)) \, h(s)  \, dx \, ds 
    \equiv \mathcal{I}_1 + \mathcal{I}_2\,. 
 \end{align}
 We estimate each of the above terms separately. Consider the term $\mathcal{I}_1$. In view of the assumption \ref{A2}, and the  property $\zeta_{\vartheta}^{''}(r) \leq \frac{K_2}{\vartheta} \textbf{1}_{  \left \{{|r| \leq \vartheta}  \right \}} $, we observe that
 \begin{align*}
     &(\vec{f}(u_1) - \vec{f}(u_2)) \cdot \nabla (u_1(s,x) - u_2(s,x)) \, \zeta_{\vartheta}^{''}(u_1 - u_2) \\
     &\leq c_f |u_1 - u_2| \ |\nabla(u_1 - u_2)| \ \frac{K_2}{\vartheta} \textbf{1}_{\{|u_1 - u_2| \leq \vartheta\}} \rightarrow 0 \ \ (\vartheta \rightarrow 0)
 \end{align*}
 for almost every $(t,x) \in D_{(0,T)}$. Moreover, 
 $$ |\vec{f}(u_1) - \vec{f}(u_2)|  \left | \nabla (u_1(s,x) - u_2(s,x)) \right | \zeta_{\vartheta}^{''}(u_1 - u_2) \leq c_f \, K_2 \left | \nabla(u_1 - u_2) \right | \in L^1(D_{(0,T)}).$$
 Thus, by dominated convergence, we conclude that $\mathcal{I}_1 \rightarrow 0$ as $\vartheta \rightarrow 0$.
 
 \vspace{0.2cm}
 We use the assumption \ref{A3} and the boundedness of $\zeta_{\vartheta}^{'}(\cdot) $ to estimate $\mathcal{I}_2$ as
 \begin{align*}
    |\mathcal{I}_2| \, & \leq C \int_{0}^{t}\int_{D} |u_1 - u_2| \, |h(s)| \, dx \, ds \leq C \int_{0}^{t} \Big( \int_{D} |u_1 - u_2|\, dx \Big) |h(s)|  \, ds\,.
    \end{align*}
Using these estimations in \eqref{eq:uni2} and sending $\vartheta \rightarrow 0$, we have
\begin{align*}
\int_{D} |u_1(t,x) - u_2(t,x)| \,dx \le C \int_0^t \left( \int_{D} |u_1(s,x) - u_2(s,x)| \,dx \right)\,|h(s)|\,ds\,.
\end{align*}
Since $h\in L^2(0,T; \R)$, by Gr{\"o}nwall's lemma, we have
 $$  \int_{D} |u_1(t,x) - u_2(t,x)|\,dx  = 0. $$
 i.e. $u_1(t,x) = u_2(t,x)$ for almost every $(t,x) \in D_{(0,T)}$.
 This complete the uniqueness proof.

\section{Proof of Large deviation principle~(Theorem \ref{thm:main})} \label{sec:LDP}
In this section, we give the proof of Theorem \ref{thm:main}. In view of Lemma \ref{lem:main}, we  wish to show that  the {\bf Condition $C$} is fulfilled by the solution $u_\eps$
of the equation \eqref{eq:ldp}. 

\subsection{Proof of condition \ref{C2}}
To prove condition \ref{C2}, since $S_M$ is weakly compact, it is enough to prove that if $h_{n} \rightharpoonup h$ in $L^2([0,T]; \mathbb{R})$ then $u_n \rightarrow u_h$ strongly in $C([0,T];L^{2}(D))\cap L^{p}([0,T];W_{0}^{1,p}(D))$ where $u_n$ and $u_h$ are the solution of \eqref{eq:skeleton} corresponding to $h_n$ and $h$ respectively.  We first derive a-priori estimate for $u_n$: there exists a constant $C$, independent of $n$ such that 
\begin{equation}
 \label{eq:ub}
    \begin{aligned}
       \sup_{s \in [0,T]} \left \|u_n(s) \right \|_{L^2(D)}^{2} + \int_{0}^{T} \left \|u_n(s) \right \|_{W_{0}^{1,p}(D)}^{p} ds    \leq C. 
    \end{aligned}
\end{equation}
Indeed, application of chain-rule, poincar{\'e} inequality and Cauchy-Schwartz inequality gives
\begin{align}
    \frac{d}{dt} \left \| u_n(t) \right \|_{L^2(D)}^{2} & = 2\left<{\rm div}_x ( |\nabla u_n|^{p-2}\nabla u_n 
    + \vec{f}(u_n)), u_n \right> + 2 \left< H(u_n)h_n(t),u_n(t)\right> \notag \\
    & \leq -C \left \| u_n \ \right \|_{W_{0}^{1,p}(D)}^{p}
    + C \left \| u_n(t) \right \|_{L^2(D)}^{2} + \left | h_n(t) \right |^2 \left \| u_n(t)) \right \|_{L^2(D)}^{2} \notag \\
    & \leq -C \left \| u_n \ \right \|_{W_{0}^{1,p}(D)}^{p} 
    + (C+ \left | h_n(t) \right |^2)\left \| u_n(t)) \right \|_{L^2(D)}^{2} \,. \notag
\end{align}
Integrating over time from $0$ to $t$, and then applying Gr{\"o}nwall's lemma together with the fact that $ \displaystyle \int_0^T h_n^2(t)\,dt \le M$, we have
$$ \sup_{n} \sup_{0\le t\le T}  \left \| u_n(t)) \right \|_{L^2(D)}^{2} \le C(T,M).$$
Moreover, one has 
\[ \sup_{n} \bigg[ \sup_{s \in [0,T]} \left \|u_n(s) \right \|_{L^2(D)}^{2} + \int_{0}^{T} \left \|u_n(s) \right \|_{W_{0}^{1,p}(D)}^{p} ds \bigg]   \leq C_{M}. \]
Furthermore,  following the proof of Lemma \ref{lem:cpt0} along with the boundedness  property of $h_n$, we get the following estimate.
\begin{align}
\sup_{n} \left\{ \left\| u_n\right\|_{W^{\alpha,p}([0,T];{W^{-1,p'}(D)})}  \right\}  < \infty \quad \text{for}~ \alpha \in (0,\frac{1}{p})\,. \label{esti:frac-Sov-sequence}
\end{align}
Therefore, in view of a-priori estimates \eqref{eq:ub} and \eqref{esti:frac-Sov-sequence} and  Lemma \ref{lem:cpt},  there exists a subsequence  of $\{u_n\}$, still denoted by $\left \{u_n \right \}$, such that $u_n \rightarrow \bar{u}$ in $L^{2}([0,T];L^{2}(D))$.

\begin{rem}\label{rem:con-1-sequence}
Results of Theorem \ref{eq:cgs1} remains valid if one replaces $\tilde{u}_{\tau}$ by $u_{n}$ and $v$ by $\bar{u}$.
\end{rem}
Thanks to Remark \ref{rem:con-1-sequence}, one can use a similar arguments  as invoked in Lemma \ref{eq:cgseq1} to conclude: for any $\phi \in W_0^{1,p}(D)$
\begin{align}\label{con:1-seq-cond-c2}
\begin{cases}
\displaystyle \underset{n\rightarrow \infty}\lim \int_{0}^{T} \left<  {\rm div}_x \big(\vec{f}(u_n(t)) -\vec{f}({\bar{u}(t)})\big), \phi \right> \, dt = 0\,, \\
 \displaystyle \underset{n\rightarrow \infty}\lim \int_{0}^{T} \left<  {\rm div}_x \big(|\nabla u_{n}(t)|^{p-2}\nabla u_{n}(t) - G(t)\big), \phi \right> \, dt = 0\,,
\end{cases}
\end{align}
for some  $ G \in L^{p'}([0,T];L^{p'}(D))^{d}$.  Observe that
\begin{align} \label{eq:cgs2.1}
     &\left |\int_{0}^{T} \left<H(u_n(t)){h}_n(t) - H(\bar{u}(t))h(t) , \phi \right> \, dt  \right | \notag \\
     &\leq \left | \int_{0}^{T} \left< \big(H(u_n(t)) - H(\bar{u}(t)) \big) {h}_n(t) , \phi \right> \, dt + \int_{0}^{T} \left< \big(h_{n}(t) - h(t) \big)H(\bar{u}(t)) , \phi \right> \, dt  \right | \notag \\
     &:= I_1(T) + I_2(T).
\end{align}
An application of H\"{o}lder's inequality, the assumption \ref{A3} and Remark  \ref{rem:con-1-sequence} reveals that
\begin{align*}
    I_1(T) & \leq \int_{0}^{T} \left \| \big(H(u_n(t)) - H(\bar{u}(t) \big) {h}_n(t) \right \|_{L^2(D)} \left \| \phi \right \|_{L^2(D)} dt \\
    & \leq \int_{0}^{T} \left \| H(u_n(t)) - H(\bar{u}(t)   \right \|_{L^2(D)} \left |{h}_n(t)  \right | \left \| \phi \right \|_{L^2(D)} dt \\
    & \leq C \int_{0}^{T} \left \| u_n - \bar{u} \right \|_{L^2(D)} \left |{h}_n(t)  \right | dt \\
     & \leq C  \bigg(\int_{0}^{T} \left \| u_n - \bar{u} \right \|_{L^2(D)}^{2} \, dt\bigg)^{\frac{1}{2}} \bigg(\int_{0}^{T}  \left |{h}_n(t)  \right |^2  \, dt \bigg)^{\frac{1}{2}} \rightarrow 0.
    \end{align*}
    Notice that the function $\displaystyle \psi(t): = \int_{D} H(\bar{u}(t,x) \phi(x) \, dx $  lies in $L^2([0,T]; \mathbb{R})$. Indeed,  thanks to the assumption \ref{A3} 
    \begin{align*}
   \|\psi\|_{L^2[0,T]}^2\leq \int_{0}^{T} \bigg(\int_{D} |H(\bar{u}(t,x)|^2 \, dx \bigg)  \bigg(\int_{D} |\phi(x)|^2 \, dx \bigg) \, dt 
     \leq C \left \| \phi \right \|_{L^2(D)}^2  \int_{0}^{T} \left \|\bar{u}(t) \right \|_{L^2(D)}^2 dt < \infty.
\end{align*}
Since $h_{n} \rightharpoonup h$  in $L^2([0,T];\mathbb{R})$, one then easily get
 $\displaystyle    I_2(T)=\int_0^T (h_n(t)-h(t))\psi(t)\,dt  \rightarrow 0$.  Therefore, we arrive at
 \begin{align}
 \lim_{n\goto \infty} \int_{0}^{T} \left<H(u_n(t)){h}_n(t) - H(\bar{u}(t))h(t) , \phi \right> \, dt =0\,. \label{con:sequence-3}
 \end{align}
 In view of Remark  \ref{rem:con-1-sequence}, \eqref{con:1-seq-cond-c2} and \eqref{con:sequence-3}, one can easily see that the limit function $\bar{u}$ satisfies weak formulation: for any $\phi \in W_0^{1,p}(D)$
 \begin{align}
\left<\bar{u}(t),\phi \right> = \left<u_{0}, \phi \right> + \int_{0}^{t} \left<  {\rm div}_x (G(s)+\vec{f}({\bar{u}(s)})), \phi \right> \, ds + \int_{0}^{t} \left< H(\bar{u}(s)h(s) , \phi \right> \, ds.\label{eq:limit-fun-skeleton-cond-c2}
\end{align}
To identify the function $G$, we first see that 
\begin{align*}
       \int_{0}^{T} \left |\int_{D} H(\bar{u}) \, \bar{u} \,dx  \right |^{2} dt & \leq \int_{0}^{T} \bigg(\int_{D} |H(\bar{u}|^2 \, dx  \bigg)  \bigg(\int_{D} |\bar{u}|^2 \, dx \bigg) \, dt \leq C \int_{0}^{T} \left \| \bar{u}(t) \right \|_{L^2(D)}^{4} dt < \infty,
\end{align*}
 and since  $h_{n} \rightharpoonup h$  in $L^2([0,T]; \mathbb{R})$, we have
\begin{align} \label{eq:lem4.4.1}
    \lim_{n\rightarrow \infty} \left | \int_{0}^{T} \left< H(\bar{u}(t)) \big(h_{n}(t) - h(t) \big) , \bar{u}(t) \right> \, dt  \right | = 0.
\end{align}
 In view of the assumption \ref{A3}, Remark  \ref{rem:con-1-sequence} and boundedness of $h_n$, we get
\begin{align}  \label{eq:lem4.4.2}
 \left |\int_{0}^{T} \left< H(u_n)h_n(t) ,u_n-\bar{u}\right> \, dt \right | & \leq \int_{0}^{T} \left \| H(u_n)  \right \|_{L^2(D)} \left | h_n \right |  \left \|u_n-\bar{u}  \right \|_{L^2(D)} \, dt \notag \\
 & \leq \bigg( \int_{0}^{T} C \left \| u_n  \right \|_{L^2(D)}^{2} \left \|u_n-\bar{u}  \right \|_{L^2(D)}^{2} \, dt \bigg)^{\frac{1}{2}} \bigg( \int_{0}^{T} \left | h_n \right |^{2} \, dt  \bigg)^{\frac{1}{2}} \notag \\
 & \leq C_M \bigg[ \sup_{s \in [0,T]} \left \| u_n  \right \|_{L^2(D)}^{2} \int_{0}^{T}  \left \|u_n-\bar{u}  \right \|_{L^2(D)}^{2} \, dt \bigg]^{\frac{1}{2}} \rightarrow 0.
  \end{align}
 \begin{align}  \label{eq:lem4.4.3}
 \left |\int_{0}^{T} \left< \big(H(u_n)-H(\bar{u})\big)h_n(t) , \bar{u}\right> \, dt \right | & \leq  \int_{0}^{T} \left \| H(u_n)-H(\bar{u})  \right \|_{L^2(D)} \left | h_n \right |  \left \|\bar{u}  \right \|_{L^2(D)} \, dt \notag \\
 & \leq \bigg( \int_{0}^{T} C \left \| u_n-\bar{u}  \right \|_{L^2(D)}^{2} \left \|\bar{u}  \right \|_{L^2(D)}^{2} \, dt \bigg)^{\frac{1}{2}} \bigg( \int_{0}^{T} \left | h_n \right |^{2} \, dt  \bigg)^{\frac{1}{2}} \notag \\
 & \leq C_M \bigg[ \sup_{s \in [0,T]} \left \| \bar{u}  \right \|_{L^2(D)}^{2} \int_{0}^{T}  \left \|u_n-\bar{u}  \right \|_{L^2(D)}^{2} \, dt \bigg]^{\frac{1}{2}} \rightarrow 0.
\end{align}
Combining \eqref{eq:lem4.4.1}, \eqref{eq:lem4.4.2} and \eqref{eq:lem4.4.3} , we have 

\begin{align}
\label{eq:proof}
  \int_{0}^{T} \left< H(u_n)h_n(t),u_n(t)\right>\, dt \rightarrow  \int_{0}^{T} \left<H(\bar{u}(t))h(t),\bar{u}(t) \right> \, dt\,.
\end{align}
Proceeding similarly as done in Subsection \ref{subsec:existence-skeleton} together with \eqref{eq:limit-fun-skeleton-cond-c2} and \eqref{eq:proof}, we can identify the function $G$ as
$$G = |\nabla \bar{u}(t)|^{p-2}\nabla \bar{u}(t)~~ \text{in}~~L^{p'}([0,T];L^{p'}(D))^d.$$ Moreover, thanks to uniqueness of solution of skeleton equation ~(cf.~Section \ref{sec:uniqueness}), $\bar{u}=u_h$ is the unique  solution of \eqref{eq:skeleton}.

 \begin{rem} \label{rem:2}
 For any $q \leq 2$, $u_n \rightarrow u_h$  in $L^{q}([0,T];L^{2}(D))$.
 \end{rem} 
To prove condition \ref{C2}, we need to show that $u_n \rightarrow u_h $  in $C([0,T];L^{2}(D))$. 
\begin{thm}
\label{thm:cgs1}
$u_n \rightarrow u_h $  in $ \mathcal{Z}:=C([0,T];L^{2}(D))\cap L^{p}([0,T];W_{0}^{1,p}(D))$.
\end{thm}

\begin{proof}
We apply chain-rule for the equation satisfied by $u_n-u_h$ and use poincar{\'e} inequality together with Young's inequality to have

\begin{align}
  &  \frac{d}{dt} \left \| u_n(t) -u_h(t) \right \|_{L^2(D)}^{2} \notag \\
  & = 2\left<{\rm div}_x ( |\nabla u_n|^{p-2}\nabla u_n) - {\rm div}_x ( |\nabla u_h|^{p-2}\nabla u_h), u_n-u_h \right> \notag \\
    &\quad \quad + 2\left< {\rm div}_x \vec{f}(u_n)-  {\rm div}_x \vec{f}(u_h), u_n-u_h \right> 
     + 2 \left< H(u_n)h_n(t)-H(u_h)h(t) ,u_n-u_h\right> \notag  \\
    & \leq -C \left \| u_n-u_h  \right \|_{W_{0}^{1,p}(D)}^{p} + \frac{\delta}{2} \left \| u_n-u_h  \right \|_{W_{0}^{1,p}(D)}^{p} + C_{\delta}\left \| u_n-u_h \right \|_{L^2(D)}^{p'} \notag  \\
    & \hspace{2cm}+ 2 \left< H(u_n)h_n(t)-H(u_h)h(t) ,u_n-u_h\right>. \notag 
\end{align}
Choose $\delta>0$ such that $ \delta < 2C,$  we obtain
\begin{align} 
    & \frac{d}{dt} \left \| u_n(t) -u_h(t) \right \|_{L^2(D)}^{2}  \notag \\
    & \leq -C \left \| u_n-u_h  \right \|_{W_{0}^{1,p}(D)}^{p}  + C_{\delta}\left \| u_n-u_h \right \|_{L^2(D)}^{p'} 
    + 2 \left< H(u_n)h_n(t)-H(u_h)h(t) ,u_n-u_h\right>. \notag
\end{align}

 Integrating over time and using Remark \ref{rem:2}, we get
\begin{align}
   & \sup_{s \in [0,T]} \left \|u_n(s)-u_h(s) \right \|_{L^2(D)}^{2} + \int_{0}^{T} \left \|u_n(s)-u_h(s) \right \|_{W_{0}^{1,p}(D)}^{p} ds \notag \\  & \leq 2 \int_{0}^{T} \left< H(u_n)h_n(t) ,u_n-u_h\right> \, dt
     - 2 \int_{0}^{T} \left< H(u_h)h(t) ,u_n-u_h\right>\,dt.
     \label{inq:con-det-}
\end{align}
Thanks to Cauchy-Schwartz inequality together with Remark \ref{rem:2}, one has
\begin{align*}
 \int_{0}^{T} \left< H(u_n)h_n(t) ,u_n-u_h\right> \, dt & \leq \int_{0}^{T} \left \| H(u_n)  \right \|_{L^2(D)} \left | h_n \right |  \left \|u_n-u_h  \right \|_{L^2(D)} \, dt \\
 & \leq \bigg( \int_{0}^{T} C \left \| u_n  \right \|_{L^2(D)}^{2} \left \|u_n-u_h  \right \|_{L^2(D)}^{2} \, dt \bigg)^{\frac{1}{2}} \bigg( \int_{0}^{T} \left | h_n \right |^{2} \, dt  \bigg)^{\frac{1}{2}}  \\
 & \leq C_M \bigg[ \sup_{s \in [0,T]} \left \| u_n(s)  \right \|_{L^2(D)}^{2} \int_{0}^{T}  \left \|u_n-u_h  \right \|_{L^2(D)}^{2} \, dt \bigg]^{\frac{1}{2}} \rightarrow 0, \notag 
 \end{align*}
 Similarly, we have
 \begin{align*}
    \int_{0}^{T} \left< H(u_h)h(t) ,u_n-u_h\right> \, dt  & \leq C_M \bigg( \int_{0}^{T}  \left \|u_n-u_h  \right \|_{L^2(D)}^{2} \, dt \bigg)^{\frac{1}{2}} \rightarrow 0\,,
\end{align*}
where in the last inequality we have used the fact that 
$h_n, h \in S_M$. Combining these results in \eqref{inq:con-det-}, we arrive at the assertion. 
\end{proof}

\subsection{Proof of condition \ref{C1}}

In this subsection, we prove condition \ref{C1}. To proceed further, we consider the following equation: for $\epsilon>0$
\begin{align}
    dv^{\epsilon}  - {\rm div}_x ( |\nabla v^{\epsilon}|^{p-2} \nabla v^{\epsilon} +\vec{f}(v^{\epsilon})) \, dt &= H(v^{\epsilon})h(t)\,dt  + \epsilon\, H(v^{\epsilon})\, dW\,, \label{eq:epsilon-main}
\end{align}
where $ \{h^\epsilon\} \subset \mathcal{A}_M$ converges to $h$ in distribution as $S_{M}$-valued random variables.  Application of Girsanov theorem and uniqueness of equation \eqref{eq:ldp}, we conclude that the equation \eqref{eq:epsilon-main} has unique solution $ \displaystyle v^{\epsilon} = \mathcal{G}^{\epsilon}\bigg( W(\cdot) + \frac{1}{{\epsilon}} \int_0^{\cdot} h^{\epsilon}(s) \, ds \bigg) $. To  prove condition \ref{C1}, it is enough to show that $v^{\epsilon} \rightarrow u_h $ in distribution  on $\mathcal{Z}$ as $\epsilon \rightarrow 0$, where $u_h$ is the unique solution of the skeleton equation \eqref{eq:skeleton}. 
 \vspace{0.2cm}
 
 We assume that $\epsilon \leq 1$.  Firstly, we show essential uniform estimation of $v^{\epsilon}$ in the upcoming Lemmas \ref{lem:est2.1} and \ref{lem:est2.2}.
 
\begin{lem} There exists a constant $C>0$ such that the following a-priori estimate holds.
\label{lem:est2.1}
 $$ \sup_{\epsilon \in (0,1]} \mathbb{E} \bigg[ \sup_{s \in [0,T]} \left \| v^{\epsilon}(s) \right \|_{L^2(D)}^{2} + \int_{0}^{T}  \left \|  v^{\epsilon}(s) \right \|_{W_{0}^{1,p}(D)}^{p} ds \bigg] \leq C. $$   
\end{lem}

\begin{proof}
By  applying It$\hat{o}$ formula, poincar{\'e} inequality and Cauchy-Schwartz inequality, we have 

\begin{align*}
     \frac{d}{dt} \left \| v^{\epsilon}(t) \right \|_{L^2(D)}^{2}  = & 2\left<{\rm div}_x ( |\nabla v^{\epsilon}(t)|^{p-2}\nabla v^{\epsilon}(t) + \vec{f}(v^{\epsilon}(t)), v^{\epsilon}(t) \right> + 2 \left< H(v^{\epsilon}(t)) \,
     h^{\epsilon}(t),v^{\epsilon}(t)\right> \notag \\
    &  \qquad + \epsilon^2  \left \| H(v^{\epsilon}(t))  \right \|_{L^2(D)}^{2}  + 2 \epsilon \left<v^{\epsilon}(t), H(v^{\epsilon}(t)) \, dW_t \right> \notag \\
     \leq & -C \left \|  v^{\epsilon}(t) \right \|_{W_{0}^{1,p}(D)}^{p}  + 2 \left \| H(v^{\epsilon}(t))  \right \|_{L^2(D)} \left | h^{\epsilon} \right | \left \| v^{\epsilon}(t)  \right \|_{L^2(D)} 
      + C {\epsilon} \left \| v^{\epsilon}(t)  \right \|_{L^2(D)}^{2} \notag \\
      &\qquad + 2 \epsilon \left< v^{\epsilon}(t), H(v^{\epsilon}(t)) \, dW(t) \right>. 
\end{align*}
 Using the assumption \ref{A3} and integrating over time, we have
\begin{align}
  & \displaystyle  \left \| v^{\epsilon}(t) \right \|_{L^2(D)}^{2} + C \int_{0}^{t} \left \|  v^{\epsilon}(s) \right \|_{W_{0}^{1,p}(D)}^{p} ds \notag \\
   & \leq \left \| u_0 \right \|_{L^2(D)}^{2} + \int_{0}^{t} C(\left|h^{\epsilon} \right|+\epsilon) \left \| v^{\epsilon}(s) \right \|_{L^2(D)}^{2} ds +  2 \epsilon \int_{0}^{t}   \left<v^{\epsilon}(s), H(v^{\epsilon}(s)) \, dW(s) \right>. \label{eq:est}
\end{align} 
Discarding nonnegative term and after taking the expectation, we get 
\begin{align} \label{eq:est2.1.0}
  \mathbb{E} \bigg[\sup_{s \in [0,t]} \left \| v^{\epsilon}(s) \right \|_{L^2(D)}^{2} \bigg]  & \leq \left \| u_0 \right \|_{L^2(D)}^{2} + \mathbb{E} \int_{0}^{t} C(\left|h^{\epsilon}(s) \right|  +\epsilon) \left \| v^{\epsilon}(s) \right \|_{L^2(D)}^{2} ds \notag \\
    & \quad +  2 \, \epsilon \, \mathbb{E} \Bigg[\sup_{s \in [0,t]} \left| \int_{0}^{s}   \left<v^{\epsilon}(r), H(v^{\epsilon}(r)) \, dW(r) \right> \right|  \Bigg]. 
\end{align}

By applying the Burkholder-Davis-Gundy inequality and  the assumption \ref{A3}, one has
\begin{align} \label{eq:est2.1.1}
    & 2 \, \epsilon \, \mathbb{E} \Bigg[\sup_{s \in [0,t]} \left| \int_{0}^{s}   \left<v^{\epsilon}(r), H(v^{\epsilon}(r)) \, dW(r) \right> \right|  \Bigg]  \notag \\
     & \leq 2 \, \epsilon \, \mathbb{E} \Bigg[ \int_{0}^{t} \left \| v^{\epsilon}(s) \right \|_{L^2(D)}^{2} \left \| H(v^{\epsilon}(s)) \right \|_{L^2(D)}^{2} ds \Bigg]^{\frac{1}{2}} \notag \\
     & \leq 2 \, \epsilon \, \mathbb{E} \Bigg[\sup_{s \in [0,t]}\left \| v^{\epsilon}(s) \right \|_{L^2(D)}^{2} C \int_{0}^{t}  \left \| v^{\epsilon}(s) \right \|_{L^2(D)}^{2} ds \Bigg]^{\frac{1}{2}} \notag \\
     & \leq 2 \, \epsilon \, \mathbb{E} \Bigg[\frac{\delta^2 }{4} \bigg(\sup_{s \in [0,t]}\left \| v^{\epsilon}(s) \right \|_{L^2(D)}^{2} \bigg)^{2} + C_{\delta^2 } \bigg(\int_{0}^{t}  \left \| v^{\epsilon}(s) \right \|_{L^2(D)}^{2} ds \bigg)^{2} \Bigg]^{\frac{1}{2}} \notag \\
     & \leq \epsilon \, \delta \, \mathbb{E} \bigg[\sup_{s \in [0,t]}\left \| v^{\epsilon}(s) \right \|_{L^2(D)}^{2} \bigg] + \epsilon \, C_{\delta^2 } \, \mathbb{E} \Bigg[ \int_{0}^{t}  \left \| v^{\epsilon}(s) \right \|_{L^2(D)}^{2} ds  \Bigg]\,
 \end{align}
 for some $\delta>0$. 
Applying H\"{o}lder's inequality and Young's inequality, we obtain 
\begin{align} \label{eq:est2.1.2}
      \mathbb{E} \left|\int_{0}^{t} \left|h^{\epsilon} \right| \left \| v^{\epsilon}(s) \right \|_{L^2(D)}^{2} \,ds \right| & \leq  \mathbb{E} \left|\bigg(\int_{0}^{t} \left|h^{\epsilon} \right|^2 \,ds \bigg)^{\frac{1}{2}} \bigg(\int_{0}^{t} \left \| v^{\epsilon}(s) \right \|_{L^2(D)}^{4} \,ds \bigg)^{\frac{1}{2}} \right| \notag \\
     &\leq C_M  \mathbb{E} \Bigg[\sup_{s \in [0,t]}\left \| v^{\epsilon}(s) \right \|_{L^2(D)}^{2} \int_{0}^{t}  \left \| v^{\epsilon}(s) \right \|_{L^2(D)}^{2} ds \Bigg]^{\frac{1}{2}} \notag \\
      & \leq C_M \delta \, \mathbb{E} \bigg[\sup_{s \in [0,t]}\left \| v^{\epsilon}(s) \right \|_{L^2(D)}^{2} \bigg] + C_M C_{\delta^2 } \mathbb{E} \Bigg[ \int_{0}^{t}  \left \| v^{\epsilon}(s) \right \|_{L^2(D)}^{2} ds \Bigg].
\end{align}
Therefore using \eqref{eq:est2.1.1} and \eqref{eq:est2.1.2} in \eqref{eq:est2.1.0}, we have for all $\epsilon \in (0,1]$ 
\begin{align*}
     \mathbb{E} \bigg[\sup_{s \in [0,t]} \left \| v^{\epsilon}(s) \right \|_{L^2(D)}^{2} \bigg] \leq \left \| u_0 \right \|_{L^2(D)}^{2} &+ C \mathbb{E} \int_{0}^{t}  \left \| v^{\epsilon}(s) \right \|_{L^2(D)}^{2} ds + (1+C_M) \, \delta \, \mathbb{E} \bigg[\sup_{s \in [0,t]}\left \| v^{\epsilon}(s) \right \|_{L^2(D)}^{2} \bigg]\\ 
     & + (1+C_M) \, C_{\delta^2 } \, \mathbb{E} \Bigg[ \int_{0}^{t}  \left \| v^{\epsilon}(s) \right \|_{L^2(D)}^{2} ds \Bigg].\\
\end{align*}
Choose $\delta$ such that $1-(1+C_M) \, \delta > 0,$.  We obtain
\begin{align*}
     \mathbb{E} \bigg[\sup_{s \in [0,t]} \left \| v^{\epsilon}(s) \right \|_{L^2(D)}^{2} \bigg] 
       \leq \left \| u_0 \right \|_{L^2(D)}^{2} + C \,\int_{0}^{t} \mathbb{E} \bigg[\sup_{s \in [0,t]}\left \| v^{\epsilon}(s) \right \|_{L^2(D)}^{2} \bigg] ds \,.
        \end{align*}
  By Gr{\"o}nwall's lemma, we have
      \begin{align} \label{eq:est2.1.3}
        \mathbb{E} \bigg[\sup_{s \in [0,T]} \left \| v^{\epsilon}(s) \right \|_{L^2(D)}^{2} \bigg] & \leq C. 
         \end{align}
      By using \eqref{eq:est2.1.2} and \eqref{eq:est2.1.3}, we have   from \eqref{eq:est}
        \begin{align}
        \mathbb{E} \Big[\int_{0}^{T}  \left \|  v^{\epsilon}(s) \right \|_{W_{0}^{1,p}(D)}^{p} ds \Big] 
        &\leq \left \| u_0 \right \|_{L^2(D)}^{2} + C(M,T)\,  \mathbb{E} \bigg[\sup_{s \in [0,T]}\left \| v^{\epsilon}(s) \right \|_{L^2(D)}^{2} \bigg]  \le C\,. \label{eq:est2.1.4}
\end{align}
Hence the assertion follows from  \eqref{eq:est2.1.3} and \eqref{eq:est2.1.4}.
\end{proof}
A-priori bounds in Lemma \ref{lem:est2.1} is not enough to show convergence of $v^\eps$ to $u_h$. For that, we need more regularity estimate of $v^\epsilon$. 
\begin{lem}\label{lem:est2.2}
Let $v^\epsilon$ be a solution of \eqref{eq:epsilon-main}. Then there exists a constant $C>0$, independent of $\epsilon>0$, such that
   $$  \mathbb{E} \Bigg[ \Big(\int_{0}^{T}  \left \|  v^{\epsilon}(s) \right \|_{W_{0}^{1,p}(D)}^{p} ds \Big)^2\Bigg] \leq C.$$
   \end{lem}
   
\begin{proof}
First we prove that 
\begin{align} \label{eq:est2.2.1}
  \sup_{s \in [0,T]}  \mathbb{E} \Big[ \left \| v^{\epsilon}(s) \right \|_{L^2(D)}^{4} \Big] & \leq C. 
\end{align}

Taking square on the both sides of \eqref{eq:est}, we have
\begin{align}
 \left \| v^{\epsilon}(t) \right \|_{L^2(D)}^{4} & \leq 3 \,\left \| u_0 \right \|_{L^2(D)}^{4} + 3
    \bigg(\int_{0}^{t} C(\left|h^{\epsilon} \right|+\epsilon) \left \| v^{\epsilon}(s) \right \|_{L^2(D)}^{2} ds \bigg)^{2}  \notag \\
    & \qquad +  12 \, \epsilon^{2} \bigg(\int_{0}^{t}   \left<v^{\epsilon}(s), H(v^{\epsilon}(s)) \, dW(s) \right> \bigg)^{2}\,. \label{eq:for-cond-c2-1}
    \end{align}
    Applying H\"{o}lder's inequality and using the fact that $h^{\epsilon} \in \mathcal{A}_M$, we have
    \begin{align}
    \label{eq:est2}
\mathbb{E} \Bigg[ \Big(\int_{0}^{t} \left|h^{\epsilon} \right| \left \| v^{\epsilon}(s) \right \|_{L^2(D)}^{2}\,ds\Big)^2 \Bigg]  & \leq C \,\mathbb{E} \Bigg[\bigg (\int_{0}^{t} \left|h^{\epsilon} \right|^{2} ds \bigg) \bigg (\int_{0}^{t} \left \| v^{\epsilon}(s) \right \|_{L^2(D)}^{4} ds \bigg) \Bigg]  \notag \\
& \leq C_M \,\mathbb{E} \Bigg[\int_{0}^{t} \left \| v^{\epsilon}(s) \right \|_{L^2(D)}^{4} ds \Bigg]. 
\end{align}
We use It{\^o}-isometry together with the assumption \ref{A3} to get
\begin{align} \label{eq:est2.2.2}
    \mathbb{E} \Bigg[ \Big(\int_{0}^{t}   \left<v^{\epsilon}(s), H(v^{\epsilon}(s)) \, dW(s) \right> \Big)^2 \Bigg] \leq C \mathbb{E} \Bigg[\int_{0}^{t} \left \| v^{\epsilon}(s) \right \|_{L^2(D)}^{4} ds \Bigg].
\end{align}
Combining  \eqref{eq:est2} and \eqref{eq:est2.2.2} in \eqref{eq:for-cond-c2-1}, we obtain, for $t\in [0,T]$
\begin{align*}
   \mathbb{E} \Big[  \left \| v^{\epsilon}(t) \right \|_{L^2(D)}^{4} \Big] \leq C\Bigg\{ \,\left \| u_0 \right \|_{L^2(D)}^{4} + \mathbb{E} \Big[\int_{0}^{t} \left \| v^{\epsilon}(s) \right \|_{L^2(D)}^{4} ds \Big]\Bigg\}. 
\end{align*}
Hence, an application of Gr{\"o}nwall's lemma yields the result \eqref{eq:est2.2.1}.
\vspace{0.2cm}

Again from \eqref{eq:est} together with \eqref{eq:est2} and \eqref{eq:est2.2.2}, we have 
\begin{align*}
    \mathbb{E} \Bigg[ \Big(\int_{0}^{T}  \left \|  v^{\epsilon}(s)  \ \right \|_{W_{0}^{1,p}(D)}^{p} ds\Big)^2 \Bigg] & 
    \le C \Bigg\{  \left \| u_0 \right \|_{L^2(D)}^{4} + \int_{0}^{T} \mathbb{E}\Big[ \left \| v^{\epsilon}(s) \right \|_{L^2(D)}^{4}\Big] ds \Bigg\} \notag \\
    & \le  C \Bigg\{  \left \| u_0 \right \|_{L^2(D)}^{4} + \sup_{t \in [0,T]} \mathbb{E} \Big[  \left \| v^{\epsilon}(t) \right \|_{L^2(D)}^{4} \Big]\Bigg\}\le C\,,
\end{align*}
where in the last inequality we have used the uniform estimate \eqref{eq:est2.2.1}. This completes the proof. 
\end{proof}

\begin{cor} \label{cor:est}
 One can use a similar calculation (under the cosmetic change) as done in Lemmas \ref{lem:est2.1} and \ref{lem:est2.2} to conclude 
$$ \sup_{\epsilon >0} \mathbb{E} \bigg[\sup_{s \in [0,T]}  \left \| v^{\epsilon}(s) \right \|_{L^2(D)}^{4} \bigg] \leq C.$$
\end{cor}

\begin{lem} \label{lem:cpt2}
The following estimation holds: for any $\alpha \in (0,\frac{1}{2p})$,
$$\sup_{\epsilon} \mathbb{E}\Big[ \left\| v^{\epsilon}\right\|_{W^{\alpha,2}([0,T];{W^{-1,p'}(D)})}\Big] < \infty. $$
\end{lem}

\begin{proof} Since $v^\epsilon$ is the unique solution of \eqref{eq:epsilon-main}, we have $\mathbb{P}$-a.s., and for all $t\in [0,T]$
\begin{align} \label{eq:cpt2.0}
    v^{\epsilon}(t) &= u_0 + \int_{0}^{t} {\rm div}_x ( |\nabla v^{\epsilon}|^{p-2}\nabla v^{\epsilon}) \, ds
    + \int_{0}^{t} {\rm div}_x \vec{f}(v^{\epsilon}) \, ds + \int_{0}^{t} H(v^{\epsilon})h^\epsilon(t) \, ds + \epsilon \int_{0}^{t} H(v^{\epsilon}) \, dW(s)  \notag\\
    &= \mathcal{J}^\epsilon_0+ \mathcal{J}_{1}^{\epsilon}(t) + \mathcal{J}_{2}^{\epsilon}(t) + \mathcal{J}_{3}^{\epsilon}(t) +\mathcal{J}_{4}^{\epsilon}(t). 
\end{align}

By Sobolev embedding $W_{0}^{1,p}(D) \hookrightarrow W^{-1,p'}(D)$ and Lemma \ref{lem:est2.1} , we get
\[ \mathbb{E} \Big[ \int_{0}^{T} \left \|v^{\epsilon}  \right \|_{W^{-1,p'}(D)}^{p} \,dt\Big] \leq \mathbb{E}\Big[ \int_{0}^{T} \left \|v^{\epsilon}  \right \|_{W_{0}^{1,p}(D)}^{p} \,dt \Big]\leq C. \]
Therefore, to complete the proof, we need to show that for any $\alpha \in (0,\frac{1}{2p})$,
\begin{align}
\mathbb{E}\Big[  \int_{0}^{T}\int_{0}^{T} \frac{\left \|\mathcal{J}_{i}^{\epsilon}(t) - \mathcal{J}_{i}^{\epsilon}(s)   \right \|_{W^{-1,p'}(D)}^{2}}{\left | t-s \right |^{1+ 2 \alpha}} \, dt \, ds\Big] \leq C \quad (0\le i\le 4)\,. \label{inq:frac-Sov-main-}
\end{align}

Since $\mathcal{J}_{0}^\epsilon$ is independent of time, \eqref{inq:frac-Sov-main-}  is satisfied by $\mathcal{J}_{0}^\epsilon$ for any  $\alpha \in (0,1)$.

W.L.O.G.\, we assume that $s < t$. By employing Jensen's inequality and \cite[Equation 4.1.14]{claud}, we have 
\begin{align}
  &  \mathbb{E}\Big[ \left \|\mathcal{J}_{1}^{\epsilon}(t) - \mathcal{J}_{1}^{\epsilon}(s)   \right \|_{W^{-1,p'}(D)}^{2}\Big]  \notag \\
   & \leq \mathbb{E} \left[ \left \| \int_{s}^{t} \Delta_p v^{\epsilon} \,dr  \right \|_{W^{-1,p'}(D)}^{2} \right]\leq \mathbb{E} \left[ \bigg(\int_{s}^{t} \left \| \Delta_p v^{\epsilon}  \right \|_{W^{-1,p'}(D)} \,dr\bigg)^{2}\right]  \notag \\
    & \leq \mathbb{E}\left[ \bigg(\int_{s}^{t} \left \|v^{\epsilon}  \right \|_{W_{0}^{1,p}(D)}^{p-1} \,dr\bigg)^2\right] \leq C (t-s)^{\frac{2}{p}} \mathbb{E} \left[\bigg(\int_{0}^{T} \left \|v^{\epsilon}  \right \|_{W_{0}^{1,p}(D)}^{p} \,dr\bigg)^{\frac{2(p-1)}{p}}\right] \notag \\
    & \leq C (t-s)^{\frac{2}{p}} \mathbb{E} \Bigg[1 +\bigg( \int_{0}^{T} \left \|v^{\epsilon}  \right \|_{W_{0}^{1,p}(D)}^{p} \,dr\bigg)^{2} \Bigg] \notag \\
    & \leq C (t-s)^{\frac{2}{p}}  + C (t-s)^{\frac{2}{p}} \mathbb{E}\left[ \bigg( \int_{0}^{T} \left \|v^{\epsilon}  \right \|_{W_{0}^{1,p}(D)}^{p} \,dr\bigg)^{2}\right]. \notag
\end{align}
Thanks to Lemma \ref{lem:est2.2}, we get
\begin{align}
   \mathbb{E} \left[\int_{0}^{T}\int_{0}^{T} \frac{\left \|\mathcal{J}_{1}^{\epsilon}(t) - \mathcal{J}_{1}^{\epsilon}(s)   \right \|_{W^{-1,p'}(D)}^{2}}{\left | t-s \right |^{1+ 2 \alpha}} \, dt \, ds\right] \leq C, \quad \, \forall \, \alpha \in (0,\frac{1}{2p}).  \notag
\end{align}

By using  the Sobolev embedding $L^2(D) \hookrightarrow W^{-1,p'}(D)$, Lemma \ref{lem:est2.1}, the assumption \ref{A2}  and Jensen's inequality, we have
\begin{align}
   \mathbb{E}\Big[ \left \|\mathcal{J}_{2}^{\epsilon}(t) - \mathcal{J}_{2}^{\epsilon}(s)   \right \|_{W^{-1,p'}(D)}^{2}\Big] & \leq  \mathbb{E} \left[ \left \| \int_{s}^{t}  {\rm div}_x \vec{f}(v^{\epsilon})  \,dr \right \|_{W^{-1,p'}(D)}^{2}\right] \leq  C \mathbb{E}\left[ \bigg(\int_{s}^{t} \left \|  v^{\epsilon}  \right \|_{L^{2}(D)} \,dr \bigg)^{2} \right] \notag \\
    & \leq C (t-s)^2 \, \mathbb{E} \Big[ \sup_{0\leq t\leq T}\left \| v^{\epsilon}(t) \right \|_{L^2(D)}^{2} \Big], \notag
\end{align}
and therefore,  for all $\alpha \in (0,\frac{1}{p})$ we have
\begin{align*}
    \mathbb{E} \left[\int_{0}^{T}\int_{0}^{T} \frac{\left \|\mathcal{J}_{2}^{\epsilon}(t) - \mathcal{J}_{2}^{\epsilon}(s)   \right \|_{W^{-1,p'}(D)}^{2}}{\left | t-s \right |^{1+2 \alpha}} \, dt \, ds\right] \leq C. 
\end{align*}
For $\mathcal{J}_{3}^{\epsilon}$, we estimate as follows. In view of the assumption \ref{A3}, Jensen's inequality and Lemma \ref{lem:est2.1}, 
\begin{align}
    \mathbb{E}\left[ \left \|\mathcal{J}_{3}^{\epsilon}(t) - \mathcal{J}_{3}^{\epsilon}(s)   \right \|_{L^{2}(D)}^{2}\right] & \leq  C (t-s) \mathbb{E} \left[\int_{s}^{t}  \left \| v^{\epsilon} \right \|_{L^2(D)}^{2} \left | h^{\epsilon}(r) \right |^{2} \,dr \right] \notag \\
    & \leq C (t-s) \, \mathbb{E} \bigg[ \sup_{0\leq t\leq T}\left \| v^{\epsilon}(t) \right \|_{L^2(D)}^{2} \int_{0}^{T}  \left | h^{\epsilon}(r) \right |^{2} \,dr \bigg] \notag \\
    & \leq C (t-s) \, \mathbb{E} \Big[\sup_{0\leq t\leq T}\left \| v^{\epsilon}(t) \right \|_{L^2(D)}^{2} \Big], \notag
\end{align}
where in the last inequality we have used that $h^\epsilon \in \mathcal{A}_M$. Thus, since $L^2(D) \hookrightarrow W^{-1,p'}(D)$, \eqref{inq:frac-Sov-main-} is satisfied by $\mathcal{J}_{3}^{\epsilon}$ for $\alpha \in (0,\frac{1}{p})$. 
\vspace{0.1cm}

It{\^o}-isometry together with the assumption \ref{A3} yields that 
\begin{align}
    \mathbb{E}\left[ \left \|\mathcal{J}_{4}^{\epsilon}(t) - \mathcal{J}_{4}^{\epsilon}(s)   \right \|_{L^{2}(D)}^{2}\right]  \leq C \mathbb{E} \bigg[ \int_{s}^{t}  \left \| v^{\epsilon}(r) \right \|_{L^2(D)}^{2} \,dr \bigg] \le C (t-s)\, \mathbb{E} \Big[\sup_{0\leq t\leq T}\left \| v^{\epsilon}(t) \right \|_{L^2(D)}^{2} \Big]. \notag
\end{align}
Hence,  for any $\alpha \in (0,\frac{1}{p})$ the estimation \eqref{inq:frac-Sov-main-} holds true for $\mathcal{J}_{4}^{\epsilon}$. 
\end{proof}

\begin{cor} The sequence 
$\left \{ \mathcal{L}(v^{\epsilon}) \right \}$ is tight on $L^{2}([0,T];L^{2}(D))$.
\end{cor}
\begin{proof}
  By Lemma \ref{lem:cpt}, we get $ L^{2}([0,T];W_{0}^{1,p}(D)) \, \cap \, W^{\alpha,2}([0,T];{W^{-1,p'}(D)})$ is compactly embedded in $L^{2}([0,T];L^{2}(D))$. So for any $N > 0$, the set
  $$M_N = \left \{ v \in L^{2}([0,T];L^{2}(D)) : \left \| v \right \|_{L^{2}([0,T];W_{0}^{1,p}(D))} + \left \|v  \right \|_{W^{\alpha,2}([0,T];{W^{-1,p'}(D)})} \leq N \right \} $$
  is a compact subset of $L^{2}([0,T];L^{2}(D))$. Observe that, thanks to Markov's inequality and Lemmas \ref{lem:est2.1} and  \ref{lem:cpt2}
  \begin{align}
      & \underset{N \rightarrow \infty}\lim \, \underset{\epsilon >0}\sup \ \mathbb{P}( v^{\epsilon} \notin M_N) \notag \\
      = & \underset{N \rightarrow \infty}\lim \, \underset{\epsilon >0}\sup \ \mathbb{P} \big(\left \| v^{\epsilon} \right \|_{L^{2}([0,T];W_{0}^{1,p}(D))} + \left \|v^{\epsilon}  \right \|_{W^{\alpha,2}([0,T];{W^{-1,p'}(D)})} > N \big) \notag \\
       = & \underset{N \rightarrow \infty}\lim \, \frac{1}{N} \ \underset{\epsilon>0}\sup \ \mathbb{E} \left[(\left \| v^{\epsilon} \right \|_{L^{2}([0,T];W_{0}^{1,p}(D))} + \left \|v^{\epsilon}  \right \|_{W^{\alpha,2}([0,T];{W^{-1,p'}(D)})}\right]=0\,. \notag
  \end{align}
  Therefore, $\left \{ \mathcal{L}(v^{\epsilon}) \right \}$ is tight on $L^{2}([0,T];L^{2}(D))$.
\end{proof}

Since $S_M$ is a Polish space, we apply Prokhorov compactness theorem to the family of laws $\left \{ \mathcal{L}(v^{\epsilon}) \right \}$ and the modified version of  Skorokhod representation theorem \cite[Theorem $C.1$]{Hausenblus-2018} to have  existence of a new probability space $(\tilde{\Omega},\tilde{\mathcal{F}},\tilde{\mathbb{P}})$, a subsequence of $\{\epsilon\}$, still denoted by $\left \{\epsilon  \right \}$, and random variables $(\tilde{v}^{\epsilon}, \tilde{h}^{\epsilon},\tilde{W}^{\epsilon})$ and $( \tilde{v}, h, \tilde{W})$  taking values in $L^{2}([0,T];L^{2}(D)) \times S_M \times \mathcal{C}([0,T]; \mathbb{R})$ such that
\begin{itemize}
    \item[i)]  $\mathcal{L}(\tilde{v}^{\epsilon}, \tilde{h}^{\epsilon},\tilde{W}^{\epsilon}) = \mathcal{L}({v}^{\epsilon},{h}^{\epsilon}, W)$ for all $\epsilon > 0$, 
    \item[ii)] $(\tilde{v}^{\epsilon}, \tilde{h}^{\epsilon},\tilde{W}^{\epsilon}) \rightarrow (\tilde{v}, h, \tilde{W})$ in $L^{2}([0,T];L^{2}(D)) \times S_M \times \mathcal{C}([0,T]; \mathbb{R})$\, $\tilde{\mathbb{P}}$-a.s. \,as $\epsilon \rightarrow 0$,
    \item[iii)] $\tilde{W}^{\epsilon}(\tilde{\omega}) = \tilde{W}(\tilde{\omega}) $ for all $\tilde{\omega} \in \tilde{\Omega}$.
\end{itemize}
Moreover, $\tilde{W}^{\epsilon}$ and $W^{\epsilon}$ are one dimensional Brownian motion over the stochastic basis 
$\big(\tilde{\Omega},\tilde{\mathcal{F}},\tilde{\mathbb{P}}, \{\tilde{\mathcal{F}}_t\}\big) $, where $ \{\tilde{\mathcal{F}}_t\}$ is the natural filtration of 
$(\tilde{v}^{\epsilon}, \tilde{h}^{\epsilon},\tilde{W}^{\epsilon}, \tilde{v}, h, \tilde{W})$, and $\tilde{v}^{\epsilon}$ is the solution to equation \eqref{eq:epsilon-main} for given $\big(\tilde{\Omega},\tilde{\mathcal{F}},\tilde{\mathbb{P}}, \{\tilde{\mathcal{F}}_t\}, \tilde{W}^{\epsilon}\big) $ corresponding to $ \tilde{h}^{\epsilon}$. Consequently,  $\{\tilde{v}^{\epsilon}\}$ satisfies the uniform estimates as mentioned in Lemma \ref{lem:est2.1}-\ref{lem:est2.2}, Corollary \ref{cor:est} and Lemma \ref{lem:cpt2}. Furthermore, one can use a similar argument as done in \cite[Lemma 3.8]{maj}(see \cite[Lemma 3.11]{majee}) to conclude the following 
\begin{align} \label{eq:oldres1}
    \tilde{v}^{\epsilon} \rightarrow \tilde{v}  \ in \ L^2 \big(\tilde{\Omega};L^{2}([0,T];L^{2}(D))\big) \ as \ \epsilon \rightarrow 0.
\end{align}
In view of the assumption \ref{A2}, \eqref{eq:oldres1} and since $\tilde{v}^\epsilon$ satisfies uniform estimates given in Lemma \ref{lem:est2.1}, there exists $ \hat{G} \in L^{p'}(\tilde{\Omega}\times D_{(0,T)})^{d}$ such that for any $\phi \in W_0^{1,p}(D)$

\begin{align}  \label{conv:c1-final-1}
\begin{cases}
\displaystyle \underset{\epsilon \rightarrow 0}\lim \ \tilde{\mathbb{E}}\left[ \int_{0}^{T}\Big|\int_0^t  \left<  {\rm div}_x \big(\vec{f}(\tilde{v}^{\epsilon}) -\vec{f}({\tilde{v}(s)})\big), \phi \right> \, ds\Big|\,dt\right] = 0\,, \\
 \displaystyle \underset{\epsilon \rightarrow 0}\lim \ \tilde{\mathbb{E}} \left[ \int_{0}^{T}\Big|\int_0^t  \left<  {\rm div}_x (|\nabla \tilde{v}^{\epsilon}(s)|^{p-2}\nabla \tilde{v}^{\epsilon}(s) - \hat{G}(s)), \phi \right> \, ds\Big|\,dt\right] = 0.
\end{cases}
\end{align}
Like in the estimation if $I_1(T)$ in \eqref{eq:cgs2.1}, we see that
\begin{align*}
& \tilde{\mathbb{E}}\left[ \int_0^T\Big|  \int_{0}^{t} \left< \big(H(\tilde{v}^\epsilon) - H(\tilde{v}(s)) \big) \tilde{h}^\epsilon(s) , \phi \right> \, ds\Big|\,dt\right] \notag \\
& \le CT \left\{ \tilde{\mathbb{E}}\Big[ \int_0^T\|\tilde{v}^\epsilon-\tilde{v}\|_{L^2(D)}^2\,dt\Big]\right\}^\frac{1}{2}\left\{ \tilde{\mathbb{E}}\Big[ \int_0^T|\tilde{h}^\epsilon(t)|^2\, dt\Big]\right\}^\frac{1}{2} \goto 0 \quad (\text{by \eqref{eq:oldres1}})\,.
\end{align*}
Observe that $\tilde{\mathbb{P}}$-a.s, $\displaystyle \tilde{\psi}(t):= \int_D H(\tilde{v}(t,x))\phi(x)\,dx$ lies in $L^2(0,T;\R)$. Since $\tilde{\mathbb{P}}$-a.s,
$h_{n} \rightharpoonup h$ in $L^2([0,T]; \mathbb{R})$, one then easily get that 
 $\displaystyle   \int_0^T (\tilde{h}^\epsilon(s)-h(s))\tilde{\psi}(s)\,ds  \rightarrow 0$. Since $\displaystyle \tilde{\mathbb{E}}\left[ \underset{0\leq t\leq T}\sup \|\tilde{v}(t)\|_{L^2(D)}^2 \right] < + \infty$, and 
 $\tilde{h}^\epsilon, h \in \mathcal{A}_M$, by Vitali convergence theorem, we get
 \begin{align*}
 \lim_{\epsilon \goto 0} \tilde{\mathbb{E}}\left[ \int_0^T\Big|  \int_{0}^{t} \left< \big(\tilde{h}^\epsilon(s) - h(s) \big) H(\tilde{v}(s,x)) , \phi \right> \, ds\Big|\,dt\right] =0\,.
 \end{align*}
As a result, we have the following convergence
\begin{align}
\underset{\epsilon \rightarrow 0}\lim \ \tilde{\mathbb{E}}\left[ \int_{0}^{T} \Big| \int_0^t \left<H(\tilde{v}^{\epsilon}(s))\tilde{h}^{\epsilon}(s) - H(\tilde{v}(s))h(s) , \phi \right> \, ds\Big|\,dt\right] = 0\,. \label{conv:c1-final-2}
\end{align}
By using BDG inequality and  the first uniform estimate in Lemma \ref{lem:est2.1}, one can easily see that
\begin{align}
& \epsilon \tilde{\mathbb{E}} \left[ \int_0^T\Big| \big\langle \int_0^t H(\tilde{v}^\epsilon(s))\,d\tilde{W}^\epsilon(s), \phi \big\rangle\Big|\,dt\right]\notag \\
& \le \epsilon  \tilde{\mathbb{E}} \left[ \int_0^T\left\| \int_0^t H(\tilde{v}^\epsilon(s))\,d\tilde{W}^\epsilon(s)\right\|_{L^2(D)}\|\phi\|_{L^2(D)}\,dt\right] \notag \\
& \le  \epsilon\|\phi\|_{L^2(D)}  \tilde{\mathbb{E}} \left[ \int_0^T\left( \int_0^T \|H(\tilde{v}^\epsilon(s))\|_{L^2(D)}^2\,ds\right)^\frac{1}{2}\,dt\right] \notag \\
& \le C \epsilon \|\phi\|_{L^2(D)}  T\, \tilde{\mathbb{E}} \left[\sup_{0\le s\le T}\|\tilde{v}^\epsilon(s)\|_{L^2(D)} \right] \goto 0 \quad (\text{as $\epsilon \goto 0$})\,. 
 \label{conv:c1-final-3}
\end{align}
One can use the convergence results in \eqref{eq:oldres1}-\eqref{conv:c1-final-3} and pass to the  limit in the equation satisfied by $\tilde{v}^{\epsilon}$ as $\epsilon \rightarrow 0$ to conclude that $\tilde{v} $  satisfies the weak form
\begin{align}
\displaystyle \left<\tilde{v}(t),\phi \right> = \left<u_{0}, \phi \right> + \int_{0}^{t} \left<  {\rm div}_x (\hat{G}(s)+\vec{f}({\tilde{v}(s)})), \phi \right> \, ds + \int_{0}^{t} \left< H(\tilde{v}(s)h(s) , \phi \right> \, ds \quad \forall~\phi \in  W_0^{1,p}(D)\,. \label{eq:final-cond-c1-skeleton}
\end{align}
Once we show that $\tilde{v}$ is indeed a solution of the the skeleton equation \eqref{eq:skeleton}, by uniqueness of skeleton equation, we get $\tilde{v} = u_h$. Since $\tilde{v}^{\epsilon}$ has the same distribution with $v^{\epsilon}$,  to prove  $v^{\epsilon} \rightarrow u_h $ in distribution as $\epsilon \rightarrow 0$, it is enough to prove that $\tilde{v}^{\epsilon} \rightarrow \tilde{v} $ in distribution on $\mathcal{Z}$ as $\epsilon \rightarrow 0$. 
\vspace{.1cm}

\noindent{\bf  Identification $\hat{G}$:} To show that $\tilde{v}$ is a solution of \eqref{eq:skeleton}, in view of  \eqref{eq:final-cond-c1-skeleton}, it is needed to show that  $\hat{G} = |\nabla \tilde{v}(t)|^{p-2}\nabla \tilde{v}(t)$ in $L^{p'}(\tilde{\Omega}\times D_{(0,T)})^{d}$. Indeed, by applying chain-rule on $\tilde{v}$ in  \eqref{eq:final-cond-c1-skeleton}, and It{\^o} formula to the functional $\|\tilde{v}^\epsilon(t)\|_{L^2(D)}^2$ in the equation satisfied by $\tilde{v}^\epsilon$, and then subtracting these resulting equations from each other, we have, after taking expectation
\begin{align}
     \tilde{\mathbb{E}} \Big[& \left \|  \tilde{v}^{\epsilon}(T) \right \|_{L^2(D)}^{2} - \left \|\tilde{v}(T) \right \|_{L^2(D)}^{2} \Big] + 2\tilde{\mathbb{E}} \Bigg[\int_{0}^{T} \int_{D}  |\nabla  \tilde{v}^{\epsilon}(t)|^{p-2}\nabla  \tilde{v}^{\epsilon}(t) \cdot \nabla  \tilde{v}^{\epsilon}(t) \, dx \, dt \Bigg]\notag  \\
    & \leq 2\tilde{\mathbb{E}} \Bigg[\int_{0}^{T} \int_{D}   \hat{G} \cdot \nabla {\tilde{v}(t)} \, dx \, dt \Bigg]+ \epsilon \ \tilde{\mathbb{E}} \Bigg[\int_{0}^{T} \left \|  H( \tilde{v}^{\epsilon}) \right \|_{L^2(D)}^{2} dt \Bigg]  \notag \\
    & + 2\tilde{\mathbb{E}} \Bigg[\int_{0}^{T} \left< H( v^{\epsilon})h^{\epsilon}(t), v^{\epsilon}(t)\right>\, dt - \int_{0}^{T} \left<H(\tilde{v}(t))h(t),\tilde{v}(t) \right> \, dt \Bigg]\,. \label{eq:conv-limit-c1}
\end{align}
Using Corollary \ref{cor:est}, one can easily show that $\tilde{\mathbb{P}}$-a.s., 
$$ \hat{\psi}(t):= \int_D H(\tilde{v}(t,x)) \tilde{v}(t,x)\,dx \in L^2(0,T;\R).$$
Since $\tilde{\mathbb{E}}\left[ \underset{0\leq t\leq T}\sup \| \tilde{v}(t)\|_{L^2(D)}^4\right]< + \infty$ and 
$\tilde{\mathbb{P}}$-a.s.,  $\tilde{h}^\epsilon \rightharpoonup h$ in $L^2(0,T;\R)$, by Vitali convergence theorem, we conclude that 
\begin{align}
    \lim_{\epsilon \rightarrow 0} \tilde{\mathbb{E}} \left | \int_{0}^{T} \left< H(\tilde{v}(t)) \big(\tilde{h}^{\epsilon}(t) - h(t) \big) , \tilde{v}(t) \right> \, dt  \right | = 0. \label{conv:weak-c1-drift-1}
\end{align}
Similar to \eqref{eq:lem4.4.2} and \eqref{eq:lem4.4.3}, one can use \eqref{eq:oldres1} together with uniform estimate in Lemma \ref{lem:est2.1} for $\tilde{v}$ and
$\tilde{v}^\epsilon$ to arrive at
\begin{align} \label{conv:weak-c1-drift-2}
\begin{cases}
\displaystyle \lim_{\epsilon \goto 0} \tilde{\mathbb{E}} \left[ \int_0^T \langle H(\tilde{v}^\epsilon) \tilde{h}^\epsilon(t), \tilde{v}^\epsilon(t)-\tilde{v}(t)\rangle\,dt\right]=0\,, \\
\displaystyle \lim_{\epsilon \goto 0} \tilde{\mathbb{E}} \left[ \int_0^T \langle  (H(\tilde{v}^\epsilon)-H(\tilde{v})) \tilde{h}^\epsilon(t), \tilde{v}(t)\rangle\,dt\right]=0\,.
\end{cases}
\end{align}
Using \eqref{conv:weak-c1-drift-1} and \eqref{conv:weak-c1-drift-2} and passing to the limit as $\epsilon \goto 0$ in \eqref{eq:conv-limit-c1}, we have
\begin{align}
\limsup_{\epsilon>0}\tilde{\mathbb{E}} \Bigg[\int_{0}^{T} \int_{D}  |\nabla  \tilde{v}^{\epsilon}(t)|^{p-2}\nabla  \tilde{v}^{\epsilon}(t) \cdot \nabla  \tilde{v}^{\epsilon}(t) \, dx \, dt \Bigg] \le \tilde{\mathbb{E}} \Bigg[\int_{0}^{T} \int_{D}   \hat{G} \cdot \nabla {\tilde{v}(t)} \, dx \, dt \Bigg]\,. \label{conv:weak-c1-drift-3}
\end{align}
Proceeding similarly as in Subsection \ref{subsec:existence-skeleton} (see also \cite[Page $194$, proof of Lemma $26$]{gv1}) and using \eqref{conv:weak-c1-drift-3}, one can identify $\hat{G}$ as  $\hat{G} = |\nabla \tilde{v}(t)|^{p-2}\nabla \tilde{v}(t)$ in $L^{p'}(\tilde{\Omega}\times D_{(0,T)})^{d}$.

\vspace{0.1cm}

As we mentioned, to finish the proof of the assertion, we need to show $\tilde{v}^{\epsilon} \rightarrow \tilde{v} $ in distribution on $\mathcal{Z}$ as $\epsilon \rightarrow 0$---which is given in the following theorem.
 \begin{thm}
$\tilde{v}^{\epsilon} \overset{D}{\rightarrow} \tilde{v} $  in $\mathcal{Z} = C([0,T];L^{2}(D))\cap L^{p}([0,T];W_{0}^{1,p}(D))$.
\end{thm} 
\begin{proof}
Since convergence in probability implies convergence in distribution, by Markov inequality the theorem will be proved if we able to show that
$$ \tilde{\mathbb{E}} \Big[ \left \|v^{\epsilon} - \tilde{v}  \right \|_{\mathcal{Z}} \Big] \rightarrow 0 \ \ \text{as} \ \epsilon \rightarrow 0. $$
 Proof of the above convergence is similar to the proof of Theorem \ref{thm:cgs1}  but in stochastic setting. Applying It{\^o} formula to the functional $u\mapsto \|u\|_{L^2(D)}^2$ on the difference equation of $\tilde{v}^\epsilon-\tilde{v}$ and then using Poincar{\'e} inequality together with Cauchy-Schwartz inequality, one has
 \begin{align}
 \tilde{\mathbb{E}} \Big[ \left \|v^{\epsilon} - \tilde{v}  \right \|_{\mathcal{Z}} \Big] & \le C \tilde{\mathbb{E}}\left[ \int_0^T \big\langle H(\tilde{v}^\epsilon) \tilde{h}^\epsilon(t)-
 H(\tilde{v}) h(t), v^{\epsilon} - \tilde{v} \big\rangle\,dt\right] + C\epsilon  \tilde{\mathbb{E}}\left[\int_0^T \|H(\tilde{v}^\epsilon(t))\|_{L^2(D)}^2\,dt\right] \notag \\
& \quad + C \epsilon \tilde{\mathbb{E}}\left[ \sup_{s\in [0,T]}\Big| \int_0^s \big\langle v^{\epsilon}(r) - \tilde{v}(r), H(\tilde{v}^\epsilon(r))\,dW(r)\big\rangle\Big|\right]
\equiv {\bf B}_1 + {\bf B}_2 + {\bf B}_3\,. \notag
 \end{align}
In view of the assumption \ref{A3} and Lemma \ref{lem:est2.1}, it is easy to see that ${\bf B}_2 \goto 0$ as $\epsilon \goto 0$, and
\begin{align*}
{\bf B}_1 \le C \left(  \tilde{\mathbb{E}}\left[ \int_0^T \|\tilde{v}^\epsilon(t)-\tilde{v}(t)\|_{L^2(D)}^2\,dt\right]\right)^\frac{1}{2} \goto 0 \quad \text{as $\epsilon \goto 0$~~~(by \eqref{eq:oldres1})}\,.
\end{align*}
We apply Burkholder-Davis-Gundy inequality, the assumption \ref{A3}, Lemma \ref{lem:est2.1} and \eqref{eq:oldres1} to obtain
\begin{align*}
{\bf B}_3  & \le C \epsilon  \tilde{\mathbb{E}}\left[ \left( \int_0^T 
 \|\tilde{v}^\epsilon(t)-\tilde{v}(t)\|_{L^2(D)}^2  \|\tilde{v}^\epsilon(t)\|_{L^2(D)}^2\,dt \right)^\frac{1}{2}\right] \notag \\
 & \le C \epsilon  \tilde{\mathbb{E}}\left[  \sup_{0\le t\le T} \|\tilde{v}^\epsilon(t)\|_{L^2(D)} \left( \int_0^T
  \|\tilde{v}^\epsilon(t)-\tilde{v}(t)\|_{L^2(D)}^2 \,dt \right)^\frac{1}{2}\right] \notag \\
  & \le  C \epsilon \left\{  \tilde{\mathbb{E}}\left[  \sup_{0\le t\le T} \|\tilde{v}^\epsilon(t)\|_{L^2(D)}^2\right]\right\}^\frac{1}{2}  \left\{  \tilde{\mathbb{E}}\left[ \int_0^T
  \|\tilde{v}^\epsilon(t)-\tilde{v}(t)\|_{L^2(D)}^2 \,dt \right]\right\}^\frac{1}{2}\goto 0 \quad \text{as $\epsilon \goto 0$}\,.
\end{align*}
This completes the proof.
\end{proof}

 \section{Transportation cost inequality: Proof of Theorem \ref{thm:maintci} } \label{sec:TCI}
Let $\mu$ be the law of the random field solution u(.,.) of \eqref{problem1.1}, viewed as probability measure on $X:=L^2([0,T];L^1(D)).$ Let  $\nu$ be a probability measure on $X$ such that $\nu  \ll \mu$. Define a new probability measure $\mathbb{P}^{*}$ on the filtered probability space $(\Omega, \mathcal{F}, \{ \mathcal{F}_t \}_{0 \leq t \leq T}, \mathbb{P})$ by 
\[d\mathbb{P}^{*} := \frac{d\nu}{d\mu}(u)d\mathbb{P}.\]
Denote the Radon-Nikodym derivative restricted on $\mathcal{F}_t$ by 
\[\textbf{M}_t := \frac{d\mathbb{P}^{*}}{d\mathbb{P}} \bigg \vert_{\mathcal{F}_{t} \,.}\]
Then $\textbf{M}_t$, t $\in$ [0,T] forms  a.s. continuous martingale with respect to given probability measure $\mathbb{P}$. 
An application of Girsanov theorem yields existence of an adapted process $g(s)$ such that 
\begin{itemize}
\item[i)]  $\mathbb{P}^{*}$ - a.s., $\displaystyle  \int_0^t g^{2}(s)\,ds \, < \, \infty $ for all t $\in$ [0,T], 
\item[ii)]  ${W}^*$ : [0,T] $\rightarrow$ $\mathbb{R}$ defined by
\[{W}^*(t) :=  W(t) - \int_0^t g(s)\,ds,\]
is a Brownian motion under the measure $\mathbb{P}^{*}$.
\end{itemize}
Moreover, thanks to \cite[Lemma $3.1$]{sarantsev-2019}, the martingale $\textbf{M}_t $ can be expressed as
\[\textbf{M}_t  =  \exp\bigg (\int_0^t g(s)\,dW(s) - \frac{1}{2}\int_0^t g^{2}(s)\,ds\bigg), \quad \mathbb{P}^{*}- \text{a.s}.\] 
and the relative entropy (Kullback information) of $\nu$ with respect to $\mu$ can be expressed in terms of the process $g$:
\begin{equation}\label{20}
   \mathcal{H}(\nu|\mu) =  \frac{1}{2}\mathbb{E}^*\left [\int_0^T g^{2}(s)\,ds\right], 
\end{equation}
where $\mathbb{E}^*$ stands for the expectation under the  new probability measure $\mathbb{P}^{*}$.
\vspace{0.1cm} 

Girsanov theorem and \cite[Theorem 2]{gv1} tells us that the solutions u(t) of equation \eqref{problem1.1} satisfies the following stochastic PDE under measure $\mathbb{P}^{*}$ 
\begin{equation}\label{eq:spde-girsanov-transform}
\begin{cases}
    du^g(t,x) - {\rm div}_x ( |\nabla u^g|^{p-2}\nabla u^g +\vec{f}(u^g)) \, dt = H(u^g(t,x))\,dW^*(t) + H(u^g(t,x))g(t)\,dt  ~~~ \text{in} ~~ D_{(0,T)}, \\
    u^g(0,.) = u_{0}(.)~~\text{on}~~ D.
    \end{cases}
\end{equation}
 We denote the solution of \eqref{eq:spde-girsanov-transform} by $u^g$. Consider the stochastic PDE
\begin{equation}\label{eq:spde-new-probability}
\begin{cases}
    du(t,x) - {\rm div}_x ( |\nabla u|^{p-2}\nabla u +\vec{f}(u)) \, dt = H(u(t,x)) \,dW^*(t)~~\text{in} ~~ D_{(0,T)}, \\
     u(0,.) = u_{0}(.)~~~\text{on} ~~ D.
    \end{cases}
\end{equation}
Equation \eqref{eq:spde-new-probability} has a  unique strong solution u(t), cf.~ \cite[Theorem 2]{gv1}. So it follows that under measure $\mathbb{P}^{*}$, the law of (u,$u^g$) forms a coupling ($\mu,\nu$). From the definition of Wasserstein distance, one has
\[ W_2(\nu, \mu)^2 \leq \mathbb{E}^*\Bigg [ \int_{0}^T \bigg( \int_{D} |u(t,x) - u^g(t,x)|\,dx \bigg)^2 dt  \Bigg ]. \]
In view of \eqref{20}, to prove the quadratic transportation cost inequality, it is sufficient to show that
\begin{equation}
\mathbb{E}^*\Bigg [ \int_{0}^T \bigg( \int_{D} |u(t,x) - u^g(t,x)|\,dx \bigg)^2 dt  \Bigg ] \leq C\mathbb{E}^*\bigg [ \int_0^T g^2(s)\,ds\bigg ]. \label{inq:final-tci}
\end{equation}

 Let us consider the convex approximation of the absolute value function $\zeta : \mathbb{R} \rightarrow \mathbb{R}$ as introduced in Subsection \ref{sec:uniqueness}. Now, we apply It{\^o} formula to the functional $\displaystyle \int_{D} \zeta_{\vartheta}(u(t) - u^g(t)) \, dx$ along with the integration by parts formula and have
 \begin{align}
& \int_{D} \zeta_{\vartheta}(u(t) - u^g(t)) dx \notag \\
& =  -  \int_{0}^{t}\int_{D} (|\nabla u |^{p-2}\nabla u - |\nabla u^g|^{p-2}\nabla u^g ) \cdot \nabla(u - u^g)(s) \, \zeta_{\vartheta}^{''}(u - u^g) \, dx \, ds  \notag \\
  & \quad -  \int_{0}^{t}\int_{D} (\vec{f}(u(s,x)) - \vec{f}(u^g(s,x))) \cdot \nabla(u - u^g)(s) \, \zeta_{\vartheta}^{''}(u - u^g) \, dx \, ds \notag \\
  &  \qquad + \ \int_{0}^{t}\int_{D} \zeta_{\vartheta}^{'}(u - u^g) H(u^g) \, g(s)  \, dx \, ds +
    \int_{0}^{t}\int_{D} \zeta_{\vartheta}^{'}(u - u^g) (H(u) - H(u^g)) \,dx \, dW (s)\notag \\
    & \qquad \quad+ \frac{1}{2}\int_0^t  \int_{D}  \zeta_{\vartheta}^{''}(u - u^g) (H(u) - H(u^g))^2\,dx\,ds\,. \label{eq:tci-1}
 \end{align}
 By using the inequality $\eqref{inequ}$ and the fact that
 and $\zeta_{\vartheta}^{''} \geq 0$,  we have, for some $C>0$
 \begin{align*}
 &- (|\nabla u|^{p-2}\nabla u - |\nabla u^g|^{p-2}\nabla u^g ) \cdot \nabla(u - u^g)(s) \zeta_{\vartheta}^{''}(u - u^g) \\
 & \leq -C \, |\nabla(u - u^g)|^{p} \, \zeta_{\vartheta}^{''}(u - u^g) \leq 0.
 \end{align*}

 Thus, equation \eqref{eq:tci-1} reduces to the following inequality
  \begin{align}
& \int_{D} \zeta_{\vartheta}(u(t) - u^g(t)) dx \notag \\
  & \le -  \int_{0}^{t}\int_{D} (\vec{f}(u(s,x)) - \vec{f}(u^g(s,x))) \cdot \nabla(u - u^g)(s) \, \zeta_{\vartheta}^{''}(u - u^g) \, dx \, ds \notag \\
  &  \qquad + \ \int_{0}^{t}\int_{D} \zeta_{\vartheta}^{'}(u - u^g) H(u^g) \, g(s)  \, dx \, ds +
    \int_{0}^{t}\int_{D} \zeta_{\vartheta}^{'}(u - u^g) (H(u) - H(u^g)) \,dx \, dW (s)\notag \\
    & \qquad \quad+ \frac{1}{2} \int_0^t \int_{D}  \zeta_{\vartheta}^{''}(u - u^g) (H(u) - H(u^g))^2\,dx\,ds\,. \label{eq:tci-2}
 \end{align}
 Squaring both sides of \eqref{eq:tci-2} and then taking expectation, we have, for $t\in [0,T]$
 \begin{align} 
   & \mathbb{E}^* \left [ \left(\int_{D} \zeta_{\vartheta}(u(t) - u^g(t)) dx\right)^2 \right] \notag \\
    & \leq  4\,\mathbb{E}^* \left [ \left(\int_{0}^{t}\int_{D} (\vec{f}(u(s,x)) - \vec{f}(u^g(s,x))) \cdot \nabla(u - u^g)(s,x) \, \zeta_{\vartheta}^{''}(u - u^g) \, dx \, ds\right)^2  \right ] \notag \\
    &\quad  +  4\,\mathbb{E}^* \left [ \left(\int_{0}^{t}\int_{D} \zeta_{\vartheta}^{'}(u - u^g) H(u^g) \, g(s)  \, dx \, ds \right)^2\right] 
     + 4 \mathbb{E}^* \left [\int_{0}^{t}\left( \int_{D} \zeta_{\vartheta}^{'}(u - u^g) (H(u) - H(u^g))\,dx\right)^2 ds  \right ] \notag \\
    &  \qquad \quad+ \mathbb{E}^*\left[ \left( \int_0^t \int_D \zeta_{\vartheta}^{''}(u - u^g) (H(u) - H(u^g))^2\,dx\,ds\right)^2\right]
    \equiv \mathcal{M}_{1}  + \mathcal{M}_2 + \mathcal{M}_3  + \mathcal{M}_4\,, \label{esti:m-sum-tci}
 \end{align}
 where in the first inequality, we have used It{\^o}-isometry. 
 Since $\big(\zeta_{\vartheta}^{''}(r)\big)^2 \leq \frac{(K_2)^2}{{\vartheta}^2}\textbf{1}_{  \left \{{|r| \leq \vartheta}  \right \}} $, we have, thanks to the assumption \ref{A2} 
 \begin{align*}
     & \big\{(\vec{f}(u) - \vec{f}(u^g)) \cdot \nabla (u(s,x) - u^g(s,x)) \zeta_{\vartheta}^{''}(u - u^g) \big\}^{2} \\
     &\leq c_f^{2} |u - u^g|^{2} \ |\nabla(u(s,x) - u^g(s,x))|^{2} \ \frac{(K_2)^{2}}{\vartheta^{2}} \textbf{1}_{\{|u - u^g| \leq \vartheta\}} \rightarrow 0 \ \ (\vartheta \rightarrow 0)
 \end{align*}
 for almost every $(s,x) \in D_{(0,T)}$. Moreover, we see that
 $$ |\vec{f}(u) - \vec{f}(u^g)|^{2}  \left | \nabla (u - u^g) \right |^{2} \big(\zeta_{\vartheta}^{''}(u - u^g) \big)^{2} \leq C\, \left | \nabla(u - u^g) \right |^{2} \in L^1(\Omega \times D_{(0,T)}).$$
 Thus, by dominated convergence theorem, we conclude that 
 \begin{align}
 \mathcal{M}_1 \rightarrow 0 \quad \text{as}\quad \vartheta \rightarrow 0. \label{esti:m1-tci}
 \end{align}
 Recalling the bound $|\zeta_\vartheta^{'}| \leq 1$ and applying H\"{o}lder's inequality together with the boundedness of $H(\cdot)$, 
 we obtain 
 \begin{align}
    |\mathcal{M}_2| \leq C \mathbb{E}^* \left [\left ( \int_{0}^{t}\int_{D} |H(u^g)|^2 \, dx \, ds \right ) \left (\int_{0}^{t}\int_{D} |g(s)|^2  \, dx \, ds \right ) \right ]  
     \leq C(T, |D|) \,\mathbb{E}^* \left[\int_{0}^{t} g^{2}(s)  \, ds\right]\,. \label{esti:m2-tci}
 \end{align}
 Thanks to the assumption \ref{A3}, and boundedness of $\zeta_\vartheta^{'}$, we estimate $\mathcal{M}_3$ as
 \begin{align}
      |\mathcal{M}_3| \, & \leq 4\, \mathbb{E}^* \left [ \int_{0}^{t} \bigg( \int_{D} |H(u) - H(u^g)|\, dx \bigg)^2 \, ds \right ] \leq C \mathbb{E}^* \left [ \int_{0}^{t} \bigg( \int_{D} |u - u^g| \, dx \bigg)^2 \, ds \right ]\,. \label{esti:m3-tci}
 \end{align}
 Again by using the assumption \ref{A3}, Jensen's inequality and \eqref{esti:approx}, we see that
 \begin{align}
 \mathcal{M}_4 \le C \mathbb{E}^*\left[ \int_0^t \int_D  \frac{(K_2)^2}{{\vartheta}^2}\textbf{1}_{  \left \{{|r| \leq \vartheta}  \right \}} |u-u^g|^4\,dx\,ds\right]
 \le C \vartheta^2 \goto 0 \quad(\vartheta \goto 0)\,. \label{esti:m4-tci}
 \end{align}
 In view of \eqref{esti:approx}, one can easily check that
 \begin{align}
  \mathbb{E}^*\left[ \left (\int_{D} |u-u^g| \,dx  \right )^2\right] & \leq   2 \mathbb{E}^*\left[ \Big(\int_{D} \zeta_{\vartheta}(u-u^g) \,dx \Big)^{2} \right] + C \vartheta^{2}\,.
  \label{esti:m0-tci}
 \end{align}
 Using \eqref{esti:m1-tci}-\eqref{esti:m0-tci} in \eqref{esti:m-sum-tci} and sending $\vartheta \goto 0$ in the resulting inequality, we have, for all $t\in [0,T]$
 \begin{align*}
  \mathbb{E}^*\left[ \left (\int_{D} |u(t)-u^g(t)| \,dx \right )^2\right] \le C  \,\mathbb{E}^* \left[\int_{0}^{t} g^{2}(s)  \, ds\right] + C\, \int_0^t  \mathbb{E}^*\left[ \left (\int_{D} |u(s)-u^g(s)| \,dx \right )^2\right] \,ds\,.
 \end{align*}
 An application of Gr{\"o}nwall's lemma then implies that
 \begin{align*}
 \sup_{0\le t\le T}   \mathbb{E}^*\left[ \left (\int_{D} |u(t)-u^g(t)| \,dx \right )^2\right]  \le C \mathbb{E}^* \left[\int_{0}^{T} g^{2}(s)  \, ds\right]\,.
 \end{align*}
Thus, there exists a constant $C>0$ such that \eqref{inq:final-tci} holds. This completes the proof of Theorem  \ref{thm:maintci}.

\vspace{1cm}

\noindent{\bf Acknowledgement:} The first author would like to acknowledge the financial support by CSIR, India.
The second author is supported by Department of Science and Technology, Govt. of India-the INSPIRE fellowship~(IFA18-MA119).

\vspace{1cm}

\noindent{\bf Data availability:}  Data sharing is not applicable to this article as no datasets were generated or analyzed during the current study.

\vspace{1cm}

\noindent{\bf Conflict of interest:}  The authors have not disclosed any competing interests.

\end{document}